\documentclass[amsart,12pt]{article}
\usepackage{amsmath,setspace,amssymb,amsfonts,amsthm,latexsym,amssymb,manyfoot}
\usepackage{mathrsfs,graphicx}
\newtheorem{theorem}{Theorem}[section]

\newtheorem{lemma}[theorem]{Lemma}
\newtheorem{proposition}[theorem]{Proposition}
\numberwithin{equation}{section}
\newtheorem{thm}{Theorem}

\newtheorem{lem}{Lemma}
\newtheorem{prop}{Proposition}

\theoremstyle{definition}

\newcommand{\fip}{\varphi}

\newcommand{\ee}{\varepsilon}
\newcommand{\en}{{\varepsilon_n}}

\numberwithin{equation}{section} \numberwithin{lem}{section}
\numberwithin{thm} {section} \numberwithin{rem} {section}
\numberwithin{prop} {section} \numberwithin{cor} {section}

\newcommand{\Om}{\Omega}
\newcommand{\om}{\omega}
\newcommand{\R}{{\mathbb R}}
\newcommand{\Rd}{\mathbb{R}^2}
\newcommand{\de}{\partial}

\newcommand{\calP}{\mathcal P}

\newcommand{\Pda}{\mathcal P(d\alpha)}

\newcommand{\vn}{v_n}
\newcommand{\un}{u_n}
\newcommand{\wn}{w_n}
\newcommand{\hn}{h_n}
\newcommand{\lambdan}{\lambda_n}
\newcommand{\cn}{c_n}
\newcommand{\varphin}{\varphi_n}
\newcommand{\kappan}{\kappa_n}
\newcommand{\zn}{z_n}
\newcommand{\bz}{\bar z}
\makeatletter \catcode`@=11
\newbox\tr@tto
\setbox\tr@tto=\hbox{{\count0=0\dimen0=-,9pt\dimen1=1,1pt\loop\ifnum
    \count0<11 \advance \count0 by1 \vrule width.51pt height\dimen1
    depth\dimen0\kern-0.17pt\advance\dimen0 by-0.05pt\advance\dimen1
    by0.1pt\repeat \loop\ifnum\count0<21\advance \count0 by1 \vrule
    width.6pt height\dimen1 depth\dimen0\kern-0.2pt \advance\dimen0
    by-0.1pt\advance\dimen1 by 0.05pt\repeat}}
\def\medint{\displaystyle\copy\tr@tto\kern-10.4pt\int}

\newcommand{\dmedint}{{}\hbox
{\vrule height 2,95pt depth -2,2pt width 6pt}\kern-0.94em }
\newcommand{\tmedint}{{}\hbox
{\vrule height 2,7pt depth -2,3pt width 5pt}\kern-8,5pt }
\newcommand{\smedint}{{}\hbox
{\vrule height 2,1pt depth -1,7 pt width 3pt}\kern-6,2pt }
\newcommand{\ssmedint}{{}\hbox
{\vrule height 1,7pt depth -1,3 pt width 3pt}\kern-6,2pt }

\newcount\ContatoreCostanti
\def\cnum#1{\global\advance\ContatoreCostanti by 1 \xdef#1{C_{\number \ContatoreCostanti}}#1}
\begin{document}
\title{Mass quantization and minimax solutions 
for Neri's mean field equation in 2D-turbulence}
\author{T.~Ricciardi\thanks{Corresponding author} and G.~Zecca\\
\small{Dipartimento di Matematica e Applicazioni ``R.~Caccioppoli"}\\
\small{Universit\`{a} di Napoli Federico II, Via Cintia - 80126 Napoli - Italy}\\
\small{E-mail: tonricci@unina.it - g.zecca@unina.it }}
\date{June 11, 2014}
\maketitle
\begin{abstract}
We study the mean field equation derived by Neri in the context of the statistical mechanics description
of 2D-turbulence, under a ``stochastic" assumption on the vortex circulations.
The corresponding mathematical problem is a nonlocal semilinear elliptic equation
with exponential type nonlinearity, containing a probability measure $\calP\in\mathcal M([-1,1])$ 
which describes the distribution of the vortex circulations.
Unlike the more investigated ``deterministic" version, 
we prove that Neri's equation may be viewed as a perturbation 
of the widely analyzed standard mean field equation, obtained
by taking $\calP=\delta_1$.
In particular, in the physically relevant case where $\calP$ is non-negatively supported
and $\calP(\{1\})>0$, we prove the mass quantization for blow-up sequences.
We apply this result to construct minimax type solutions
on bounded domains in $\Rd$ and on compact 2-manifolds without boundary.
\end{abstract}
\section{Introduction and statement of the main results.}
\label{sec:intro}
We are interested in the mean field equation derived by Neri~\cite{ne} in the
context of the statistical mechanics description of two-dimensional turbulence.
Such an approach was introduced in 1949 by Onsager
in the pioneering article~\cite{Onsager}, 
with the aim of explaining the formation of stable large-scale vortices,
and is still of central interest in fluid mechanics, see \cite{BouchetVenaille, Chavanis}.
\par
Neri's mean field equation \cite{ne} is derived under the ``stochastic" assumption that the point vortex
circulations are independent identically distributed random variables,
with probability distribution $\calP$.
On a bounded domain $\Omega\subset\Rd$, Neri's equation takes the form:
\begin{equation}
\label{eq:Neri}
\left\{\begin{split}
- \Delta u =&\lambda\,\int_{[-1,1]}\frac{\alpha e^{\alpha u}\,\calP (d\alpha)}
{\iint_{[-1,1]\times \Om} e^{\alpha u}\,\calP (d\alpha)dx}&&\text{ in }\Omega \\
 u=&0&&\mbox{ on }\de \Om.
\end{split}\right.
\end{equation}
Here, $u$ denotes the stream function, $\lambda>0$ is a constant related to the inverse temperature,
$dx$ is the volume element on $\Om$
and $\calP$ is a Borel probability measure defined on $[-1,1]$ denoting
the distribution of the circulations.
We note that when $\calP(d\alpha)=\delta_1(d\alpha)$, equation~\eqref{eq:Neri} 
reduces to the standard mean field equation
\begin{equation}
\label{eq:standardmfe}
\left\{\begin{split}
- \Delta u =&\lambda\,\frac{e^{u}}
{\int_{\Om} e^{u}\,dx}&&\text{ in }\Omega \\
 u=&0&&\mbox{ on }\de \Om.
\end{split}\right.
\end{equation}
Equation~\eqref{eq:standardmfe} has been extensively analyzed in the context of
turbulence in \cite{CLMP, Kiessling}. It is also relevant in many other contexts,
including the Nirenberg problem in differential geometry and the desciption of chemotaxis
in Biology. See, e.g., \cite{Lin,s1book}
and the references therein.
\par
On the other hand, a ``deterministic" assumption on the distribution of the vortex 
circulations yields the following similar equation
\begin{equation}
\label{SawadaSuzuki}
\left\{\begin{split}
- \Delta u =&\lambda\,\int_{-1}^1\frac{\alpha e^{\alpha u}\,\calP (d\alpha)}
{\int_{\Om} e^{\alpha u}\,dx}&&\text{ in }\Omega \\
 u=&0&&\mbox{ on }\de \Om,
\end{split}\right.
\end{equation}
see \cite{SawadaSuzuki}.
An unpublished informal
version of \eqref{SawadaSuzuki} was actually obtained by Onsager himself, see \cite{ES}.
Equation~\eqref{SawadaSuzuki} also includes the standard mean field equation~\eqref{eq:standardmfe}
as a special case.
\par
Thus, it is natural to ask for which probability measures
$\calP$ the results known for \eqref{eq:standardmfe} may be extended
to equations \eqref{eq:Neri} and \eqref{SawadaSuzuki},
and whether or not the equations \eqref{eq:Neri} and \eqref{SawadaSuzuki}
share similar properties.
We note that, up to a rescaling with respect to $\alpha$, we may assume without loss of generality
that
\begin{equation}
\label{assumpt:suppP1}
\mathrm{supp}\calP\cap\{-1,1\}\neq\emptyset.
\end{equation}
The ``deterministic" equation~\eqref{SawadaSuzuki} has been considered in \cite{ORS,RS}
from the point of view of the blow-up of solution sequences, and the optimal Moser-Trudinger constant.
Liouville systems corresponding to discrete versions of \eqref{SawadaSuzuki} have been
widely considered, see, e.g., \cite{CK, CSW,os3} and the references therein. In these articles, it appears that
equation~\eqref{SawadaSuzuki} behaves quite differently from \eqref{eq:standardmfe},
particularly from the point of view of the corresponding optimal Moser-Trudinger constant,
whose rather complicated expression depending on $\calP$ was recently determined in \cite{RS}.
\par
On the other hand, fewer mathematical results are available for \eqref{eq:Neri}. 
In \cite{s1book} it is conjectured as an open problem that the optimal Moser-Trudinger constant
for \eqref{eq:Neri} could depend on $\calP$.
However, in \cite{RZ} we proved that this is not the case.
More precisely, we showed that, assuming \eqref{assumpt:suppP1},
the optimal Moser-Trudinger constant for the corresponding variational functional
$J_\lambda$, given by
\begin{equation}
\label{eq:Nerifunctional}
J_\lambda(u)=\frac{1}{2}\int_\Om|\nabla u|^2\,dx-\lambda\log\left(\iint_{[-1,1]\times \Om} e^{\alpha u}\,\calP (d\alpha)dx\right)
\end{equation}
coincides with the optimal Moser-Trudinger constant for the ``standard" case $\mathcal P(d\alpha)=\delta_1(d\alpha)$.
In other words, $J_\lambda$
is bounded below if and only if $\lambda\leqslant8\pi$.
In this article  we further confirm the significant differences between \eqref{eq:Neri} and \eqref{SawadaSuzuki}, which could in principle provide a criterion to identify the more suitable model among 
\cite{ne} and \cite{SawadaSuzuki} to describe turbulent flows with variable intensities. More precisely,  we prove that,
under the additional assumption
\begin{align}
\label{calH}
&\mbox{supp} \mathcal P \subseteq [0,1],
&&\calP(\{1\})>0,
\end{align}
corresponding to the case where the vortices have the same orientation, as well as a non-zero probability
of unit circulation, the blow-up masses have quantized values $8\pi m$, $m\in\mathbb N$. 
To this end, we follow the elegant complex analysis approach in \cite{ye}. 
We note that the sinh-Poisson case $\mathcal P=\tau\delta_1+(1-\tau)\delta_{-1}$,
$\tau\in[0,1]$ was studied in \cite{RZ2}.
\par
Concerning the existence problem for equation~\eqref{eq:Neri}, we note that
Neri himself derived an existence result in the subcritical case 
$\lambda<8\pi$ by minimizing the functional~\eqref{eq:Nerifunctional}.
We shall here apply our blow-up results to prove the existence of saddle-type solutions in the supercritical case $\lambda>8\pi$,
following some ideas in  \cite{DJLW,StTa}. 
Such approaches employ the ``Struwe monotonicity trick" \cite{St} and
an improved Moser-Trudinger inequality in the sense of Aubin~\cite{Aub}.
\par
We now state our main results.
Let $\Omega\subset\Rd$ be a bounded domain and let $g$ be a metric on $\Om$.
We consider solution sequences to Neri's equation in the following  ``local'' form:
\begin{equation}
\label{eq:Nerilocal}
-\Delta_g\un =\lambdan\,\int_{[0,1]}\frac{\alpha e^{\alpha\un}\,\calP (d\alpha)}
{\iint_{[0,1]\times \Om} e^{\alpha\un}\,\calP (d\alpha)dv_g}+c_n\qquad\text{ in }\Omega 
\end{equation}
where $c_n\in\mathbb R$, $dv_g$ denotes the volume element 
and $\Delta_g$ denotes the Laplace-Beltrami operator. 
As usual, for every solution sequence $\un$ we define the blow-up set
\[
\mathcal S=\{x\in\Om\ s.t.\ \exists x_n\to x:\ \un(x_n)\to+\infty\}.
\]
\begin{thm}[Mass quantization]
\label{thm:massquantization}
Assume \eqref{calH}.
Let $\un$ be a solution sequence to \eqref{eq:Nerilocal} with $\lambdan\to\lambda_0$
and $c_n\to c_0$.
Then there exists a subsequence, still denoted $\un$, such that exactly one of the
following holds:
\begin{enumerate}
  \item [(i)]
$\un$ converges locally uniformly to a smooth solution $u_0$ for \eqref{eq:Nerilocal};
\item[(ii)]
$\un\to-\infty$ locally uniformly in $\Om$; 
\item[(iii)] The blow-up set $\mathcal S$ is finite and non-empty.
Denoting $\mathcal S=\{p_1,\ldots,p_m\}$, there holds 
\[
\lambdan\,\int_{[0,1]}\frac{\alpha e^{\alpha\un}\,\calP (d\alpha)}
{\iint_{[0,1]\times \Om} e^{\alpha\un}\,\calP (d\alpha)dv_g}\,dv_g
\stackrel{*}{\rightharpoonup}\sum_{1=1}^mn_i\delta_{p_i}+r(x)\,dv_g
\]
weakly in the sense of measures, for some $n_i\geqslant4\pi$, $i=1,\ldots,m$, and $r\in L^1(\Om)$.
In $\Om\setminus\mathcal S$ either $\un$ is locally bounded, or $\un\to-\infty$
locally uniformly.
\par
If $\un$ is locally bounded in $\Om\setminus\mathcal S$, then 
\[
\iint_{[0,1]\times\Om} e^{\un}\,dv_g\to+\infty,
\]
$r\equiv0$, $n_i=8\pi$ for all $i=1,\ldots,m$
and there exists $u_0\in W_{\textrm{loc}}^{1,q}(\Om)$
for any $q\in[1,2)$
such that $\un\to u_0$ in $W_{\textrm{loc}}^{1,q}(\Om)$
and locally uniformly in $\Om\setminus\mathcal S$.
The function $u_0$ is of the form
\[
u_0(x)=\sum_{j=1}^m\frac{1}{4}\log\frac{1}{d_g(x,p_i)}+b(x),
\]
where $b$ is locally bounded in $\Om$.
\end{enumerate}
\end{thm}
As an application of Theorem~\ref{thm:massquantization},
we derive the existence of minimax type solutions in the supercritical range $\lambda>8\pi$.
Our first existence result is derived in the spirit of \cite{DJLW}.
\begin{thm}
[Existence of a minimax solution on annulus-type domains]
\label{main-link}
Let $\Om\subset \Rd$ be a smooth, bounded domain whose complement
contains a bounded region and assume \eqref{assumpt:suppP1}.  
Then, problem~ \eqref{eq:Neri} admits a solution $u\in H_0^{1,2}(\Om)$ for almost every $\lambda\in(8\pi,16\pi)$.
Furthermore, if $\calP$ satisfies \eqref{calH}, then  \eqref{eq:Neri} admits a solution $u\in H_0^{1,2}(\Om)$
for all $\lambda\in(8\pi,16\pi)$.
\end{thm}
We also consider solutions to Neri's equation on a compact orientable Riemannian
surface without boundary $M$.
On the manifold~$M$, the corresponding problem is given by
\begin{equation}
\label{EL0}
\left\{
\begin{split}
-\Delta_g v=& \lambda\,\int_{[-1,1]}\frac{\alpha (e^{\alpha v}- \frac{1}{|M|} \int_M e^{\alpha v}\,dv_g)\,}
{\iint_{[-1,1]\times M} e^{\alpha v}\,\calP (d\alpha)\,dv_g}\calP (d\alpha)&&\text{ in }M  \\
\int_{M }v\,dv_g=&0,
\end{split}%
\right.   
\end{equation}
Here, $g$ denotes the Riemannian metric on $M$, $dv_g$ denotes the volume element 
and $\Delta_g$ denotes the Laplace-Beltrami operator.
We note that the proof of Theorem~\ref{main-link} may be adapted to problem~\eqref{EL0} provided
$M$ has genus greater than or equal to one. However, in this case it is not clear
in general
whether or not the solution obtained is distinct from the trivial solution $u\equiv0$. 
See \cite{COS} for some results and conjectures in this direction.
On the other hand, a nontrivial solution in the supercritical range of $\lambda$ for general
manifolds may be obtained by the argument introduced in \cite{StTa}.
In order to state our second existence result, we denote by $\mu_1(M)$ the first non-zero eigenvalue of $\Delta_g$, namely 
\begin{equation}
\label{mu1}
\mu _{1}(M ):=\inf_{\phi\neq 0,\text{ }\phi\in \mathcal E}\frac{\int_{M
}\left\vert \nabla  \phi \right\vert ^{2}}{\int_{M }\phi ^{2}},
\end{equation}
where $\mathcal E=\{v\in H^1(M):\ \int_Mv\,dv_g=0\}$.
We prove:
\begin{thm}[Mountain-pass solution on manifolds]
\label{main-mp}
Let $\mathcal{P}$ satisfy \eqref{assumpt:suppP1}
and let $M$ be such that 
$\frac{\mu _{1}(M )\left\vert M \right\vert}{\int_{I}\alpha ^{2}\mathcal{P}(d\alpha) }>8\pi$. 
Then, for almost every  $\lambda\in \left(8\pi,
\frac{\mu _{1}(M )\left\vert M \right\vert}{\int_{I}\alpha ^{2}\mathcal{P}(d\alpha) } \right)$ there exists a non-trivial 
solution to problem~\eqref{EL0}.
Furthermore, if $\calP$ satisfies \eqref{calH} and if $M $ is such that 
$\frac{\mu _{1}(M )\left\vert M \right\vert}{\int_{I}\alpha ^{2}\mathcal{P}(d\alpha) }\in(8\pi,16\pi)$,
then problem~(\ref{EL0}) admits a non-trivial solution  
for every $\lambda \in\left(8\pi,
\frac{\mu _{1}(M )\left\vert M \right\vert}{\int_{I}\alpha ^{2}\mathcal{P}(d\alpha) } \right)$.
\end{thm}
We organize this article as follows.
In Section~\ref{sec:compactness} we prove the mass quantization
for blow-up sequences, as stated in  Theorem~\ref{thm:massquantization}.
In Section~\ref{sect-mt} we establish an improved Moser-Trudinger inequality for
the Neri functional~\eqref{eq:Nerifunctional}, on the line of Aubin~\cite{Aub}.
In Section~\ref{sec:struwe} we derive the Struwe's Monotonicity trick,
originally introduced in \cite{St} to construct bounded Palais-Smale sequences,
in a form suitable for application to both Theorem~\ref{main-link}
and Theorem~\ref{main-mp}.
Our version of the monotonicity trick is therefore somewhat more general than the versions in \cite{DJLW,StTa}.
It should be mentioned that the monotonicity trick itself has
attracted a considerable interest, and very general versions 
have been recently derived in \cite{JJ, JT, Sq}. Here, we choose to derive Struwe's argument in a specific form
best suited 
to our applications, which also allows us to explicitly exhibit the corresponding deformations. 
Applying these results, in Section~\ref{sec:minimax} we prove Theorem~\ref{main-link}
and in Section~\ref{sec:mp} we prove Theorem~\ref{main-mp},
suitably adapting the ideas in \cite{DJLW} and \cite{StTa} respectively, 
in order to
take into account of the probability measure $\mathcal P$.
\paragraph*{Notation}
Here and below, $\Om\subset\Rd$ always denotes a smooth bounded domain and
$M$ always denotes a compact Riemannian 2-manifold without boundary.
All integrals are taken with respect to the standard Lebesgue measure.
When the integration measure is clear from the context,
we may omit it for the sake of clarity.
We denote by $C>0$ a general constant whose actual value may vary from line to line.
For every real number $t$ we set  $t^+=\max\{0,t\}$.
\section{Mass quantization and proof of Theorem~\ref{thm:massquantization}}
\label{sec:compactness}
In this section we analyze the blow-up behavior of solution sequences for \eqref{eq:Nerilocal}.
Unlike the approaches in \cite{NS, ORS, RZ}, where the cases of Dirichlet boundary conditions 
and of 
compact manifolds  without boundary are considered, we establish our blow-up results
in a more flexible local form, in the spirit of \cite{bm}.
We prove the mass quantization extending the complex analysis approach introduced in \cite{BBH, ye}.
\par
In order to prove Theorem~\ref{thm:massquantization}, we begin by establishing
a Brezis-Merle type alternative for equations with probability measures.
More precisely, let $\Om\subset\Rd$ be a bounded domain.
We consider solution sequences to the equation
\begin{equation}
\label{eq:Nerinobdry}
-\Delta\un=\int_{[0,1]}V_{\alpha,n}(x)\,e^{\alpha\un}\,\Pda
+\varphin\qquad\mathrm{in\ } \Om.
\end{equation}
where $\varphin \in L^\infty(\Om).$
We begin by proving the existence of a ``minimal mass" necessary for blow-up to occur.
\begin{proposition}[Brezis-Merle alternative]
\label{prop:localblowup}
Assume $\calP$ satisfies \eqref{calH} and suppose the following bounds hold:
\begin{enumerate}
  \item [(1)]
 $0\leqslant V_{\alpha,n}\leqslant C$, $\|\varphin\|_{L^\infty(\Om)}\leqslant C$;
  \item [(2)]
 $\iint_{[0,1]\times\Om}V_{\alpha,n}(x)\,e^{\alpha\un}\,\Pda dx\leqslant C$
\item [(3)]
$\|\un^+\|_{L^1(\Om)}\leqslant C$.
\end{enumerate}
Then, exactly one of the following alternatives holds true:
\begin{enumerate}
  \item[(i)] $\un$ converges locally uniformly in $\Om$ to a bounded function $u_0$;
  \item[(ii)] $\un\to-\infty$ locally uniformly in $\Om$;
  \item[(iii)] There exists a finite set $\mathcal S=\{p_1,\ldots,p_m\}\subset\Om$
such that 
\[
\int_{[0,1]}V_{\alpha,n}(x)\,e^{\alpha\un}\,\Pda\,dx
\stackrel{*}{\rightharpoonup}\sum_{i=1}^mn_i\delta_{p_i}+r(x)\,dx,
\]
weakly in the sense of measures,
with $n_i\geqslant4\pi$, $i=1,\ldots,m$ and $r\in L^1(\Om)$.
Moreover, $\un^+$ is locally uniformly bounded in $\Om\setminus\mathcal S$.
\par
If $\un$ is also locally uniformly bounded from below in $\Om\setminus\mathcal S$,
then $\int_{B_\rho(p_j)}e^{\un}\,dx\to+\infty$ for any ball $B_\rho(p_j)\subset\Om$.
In particular, we have
\[
\iint_{[0,1]\times\Om}e^{\alpha\un}\,\Pda dx\to+\infty.
\]
\end{enumerate}
\end{proposition}
Once Proposition~\ref{prop:localblowup} is established, setting
\[
V_{\alpha,n}(x)={\alpha }^{-1}{\iint_{[0,1]\times\Om}e^{\alpha\un}\,\Pda dx},
\]
we readily derive alternatives
(i)--(ii) and the first part of alternative~(iii) in Theorem~\ref{thm:massquantization}.
In order to complete the proof of alternative~(iii) in Theorem~\ref{thm:massquantization}, 
we need to show that if
$\mathcal S\neq\emptyset$, then $r\equiv0$ and $n_i=8\pi$ for all $i=1,\ldots,m$.
To this end, we prove that along a blow-up sequence \eqref{eq:Nerilocal} is 
equivalent to a nonlinear equation to which the complex analysis argument 
in \cite{ye} may be applied.
More precisely, we show:
\begin{proposition}
\label{prop:reductiontoYe}
Let $(\lambdan,\un)$ be a solution sequence for \eqref{eq:Nerilocal}
with $\lambdan\to\lambda_0$.
Assume that $\mathcal S\neq\emptyset$ and $\un\geqslant-C$ for some $C>0$.
Then $\un$ satisfies the equation
\begin{equation}
\label{eq:Ye}
-\Delta_g\un=\kappan f(\un)+c_n\qquad\textrm{in}\ \Om,
\end{equation}
for some $f(t)=e^t+o(e^t)$ as $t\rightarrow +\infty$, and for some $\kappan\to0$.
Passing to a subsequence, we have $\un\to u_0$ in $W_{\textrm{loc}}^{1,q}(\Om)$ for all $q\in[1,2)$,
where
\begin{equation*}
u_0(x)=\frac{1}{4}\sum_{i=1}^n\log\frac{1}{d_g(x,p_i)}+b(x)
\end{equation*}
for some $b\in L_{\textrm{loc}}^\infty(\Om)$.  
At every blow-up point $p_i$, $i=1,\ldots,m$, we have 
\begin{equation}
\label{eq:blowuppointpositions}
\nabla\left(b(p_i)+\frac{1}{4}\sum_{j\neq i}\log\frac{1}{d_g(p_i,p_j)}\right)=-\nabla\xi(p_i)
\end{equation}
where $\xi$ is the conformal factor defined by
$g=e^{\xi(x)}(dx_1^2+dx_2^2)$. The blow-up masses satisfy $n_{p_i}=8\pi$, $i=1,\ldots,m$ and
\begin{align*}
&\kappan f(\un)\stackrel{*}{\rightharpoonup}8\pi\sum_{i=1}^n\delta_{p_i},
\end{align*}
weakly in the sense of measures.
\end{proposition}
Theorem~\ref{thm:massquantization} will follow as a direct consequence of Proposition~\ref{prop:localblowup}
and Proposition~\ref{prop:reductiontoYe}.
We proceed towards the proof of Proposition~\ref{prop:localblowup}.
We recall the following well-known basic estimate.
\begin{lemma}[\cite{bm}]
\label{lem:bmestimate}
Let $\Om\subset\Rd$ be a bounded domain and let $-\Delta u=f$ in $\Om$,
$u=0$ on $\partial\Om$, with $\|f\|_{L^1(\Om)}<+\infty$.	
Then, for any $\eta\in(0,1)$ we have
\[
\int_\Om \exp\left\{\frac{4\pi(1-\eta)}{\|f\|_{L_1(\Om)}}|u|\right\}\,dx\leqslant\frac{\pi}{\eta}(\mathrm{diam}\,\Om)^2.
\]
\end{lemma}
Using Lemma~\ref{lem:bmestimate} we can show the existence of a minimal mass for blow-up for 
equations containing a probability measure.
Let $\Om\subset\Rd$ be a bounded domain. Let $\un$ be a solution to \eqref{eq:Nerinobdry} and let
\[
\nu_n=\iint_{[0,1]\times\Om}V_{\alpha,n}(x)e^{\alpha\un}\,\Pda dx.
\]
In view of assumption~(2) in Proposition~\ref{prop:localblowup}, 
passing to a subsequence there exists $\nu_0\in\mathcal M(D)$
such that $\nu_n\stackrel{*}{\rightharpoonup}\nu_0$ weakly in the sense of measures.
The next lemma states that a minimal mass $4\pi$ is necessary for blow-up to occur.
\begin{lemma}[Minimal mass for blow-up]
\label{lem:minmass}
Let $\un$ be a solution to \eqref{eq:Nerinobdry}.
Suppose $\|V_{\alpha,n}\|_{L^\infty(\Om)}\leqslant C$, $\|\un^+\|_{L^1(\Om)}\leqslant C$,
$\iint_{[0,1]\times\Om}V_{\alpha,n}(x)e^{\alpha\un}\,\Pda dx\leqslant C$, $\|\varphin\|_{L^\infty(\Om)}\leqslant C$
and suppose $x_0\in \Om$ is such that $\nu_0(\{x_0\})<4\pi$.
Then, there exists $\rho_0>0$ such that $\|\un^+\|_{L^\infty(B_{\rho_0}(x_0))}\leqslant C$.
\end{lemma}
\begin{proof}
Let $\ee_0, \rho_0>0$ be such that 
$\iint_{[0,1]\times B_{\rho_0}(x_0)}|V_{\alpha,n}(x)|e^{\alpha\un}\,\Pda dx\leqslant4\pi(1-2\ee_0)$
and $\|\varphin\|_{L^\infty(B_{\rho_0}(x_0))}|B_{\rho_0}(x_0)|\leqslant4\pi\ee_0$.
Let $\wn$ be defined by
\begin{equation}
\label{eq:wn}
\left\{
\begin{split}
-\Delta\wn=&\int_{[0,1]}V_{\alpha,n}(x)e^{\alpha\un}\,\Pda+\varphin
&&\mathrm{in\ }B_{\rho_0}(x_0)\cr
\wn=&0&&\mathrm{on\ }\partial B_{\rho_0}(x_0).
\end{split}
\right.
\end{equation}
Setting $\psi_n=\int_{[0,1]}V_{\alpha,n}(x)e^{\alpha\un}\,\Pda+\varphin$,
we have $\|\psi_n\|_{L^1(B_{\rho_0}(x_0))}\leqslant4\pi(1-\ee_0)$.
By elliptic estimates, $\|\wn\|_{L^1(B_{\rho_0}(x_0))}\leqslant C$.
In view of Lemma~\ref{lem:bmestimate} we derive for every $\eta\in(0,1)$ that
\[
\int_{B_{\rho_0}(x_0)}\exp\left\{\frac{1-\eta}{1-\ee_0}|\wn|\right\}\leqslant\frac{4\pi\rho_0^2}{\eta}.
\]
Choosing $\eta<\ee_0^2$, we find $\|\psi_n\|_{L^{1+\ee_0}(B_{\rho_0}(x_0))}\leqslant C$.
\par
On the other hand, the function
$\hn:=\un-\wn$ is harmonic in $B_{\rho_0}(x_0)$ and 
\[
\|\hn^+\|_{L^1(B_{\rho_0}(x_0))}\leqslant\|\un^+\|_{L^1(\Om)}+\|\wn\|_{L^1(B_{\rho_0}(x_0))}\leqslant C.
\]
Hence, the mean value theorem implies that $\|\hn^+\|_{L^\infty(B_{\rho_0/2}(x_0))}\leqslant C$.
Inserting into \eqref{eq:wn}, we find $\|\wn\|_{L^\infty(B_{\rho_0}(x_0))}\leqslant C\|\psi_n\|_{L^{1+\ee_0}(B_{\rho_0}(x_0))}\leqslant C$.
Finally, we have
\[
\|\un^+\|_{L^\infty(B_{\rho_0/2}(x_0))}\leqslant\|\hn^+\|_{L^\infty(B_{\rho_0/2}(x_0))}+\|\wn\|_{L^\infty(B_{\rho_0}(x_0))}\leqslant C.
\]
Since $\rho_0$ is arbitrary, the asserted local uniform boundedness of $\un$ is established.
\end{proof}
\begin{proof}[Proof of Proposition~\ref{prop:localblowup}]
In view of Harnack's inequality, if $\mathcal S=\emptyset$ then (i) or (ii)
hold.
Therefore, we assume $\mathcal S\neq\emptyset$.
By Harnack's inequality, either $\un$ is locally bounded from below in $\Om\setminus\mathcal S$,
or $\un\to-\infty$ locally uniformly in $\Om\setminus\mathcal S$. Let $p_i\in\mathcal S$ and let $\rho>0$ be such that
$\overline{B_\rho(p_i)}\cap\mathcal S=\{p_i\}$.
We assume that $\un\geqslant-C$ on $\partial B_\rho(p_i)$. Similarly as in \cite{bm}, we define
\begin{equation*}
\left\{
\begin{split}
-\Delta\zn=&\int_{[0,1]}V_{\alpha,n}(x)\,e^{\alpha\un}\,\Pda
+\varphin&&\textrm{in\ } B_\rho(p_i)\\
\zn=&-C&&\textrm{on\ }\partial B_\rho(p_i).
\end{split}
\right.
\end{equation*}
Then, $\un\geqslant\zn$ in $\overline{B_\rho(p_i)}$.
On the other hand, $\zn\to z$ in $W^{1,q}(B_\rho(p_i))$ for all
$q\in[1,2)$, with $z\geqslant \log |x-p_i|^{-2}-C$.
By Fatou's lemma, we conclude that
$\int_{B_\rho(p_i)}e^{\un}\,dx\to+\infty$.
In view of assumption~\eqref{calH} we derive in turn that
\begin{equation}
\label{lim:iinttoinfty}
\iint_{[0,1]\times\Om} e^{\alpha\un}\,\Pda dx\geqslant\calP(\{1\})\int_\Om e^{\un}\,dx\to+\infty.
\end{equation}
\end{proof}
\begin{proof}[Proof of Propostion~\ref{prop:reductiontoYe}]
Since $g$ is given in isothermal coordinates, namely $g=e^{\xi(x_1,x_2)}(dx_1^2+dx_2^2)$,
then \eqref{eq:Nerilocal} takes the form
\[
-\Delta\un=\lambdan e^\xi\int_{[0,1]}\frac{\alpha e^{\alpha\un}}{\iint_{[0,1]\times\Om}e^{\alpha\un}}\,\Pda+e^\xi\cn.
\]
We apply Proposition~\ref{prop:localblowup}--(iii) 
with $V_{\alpha,n}=e^\xi\lambdan\alpha(\iint_{[0,1]\times\Om}e^{\alpha\un}\,\Pda\,dv_g)^{-1}$
and $\varphin=e^\xi\cn$.
Since $\un\geqslant-C$, we conclude that
\[
\iint_{[0,1]\times\Om}e^{\alpha\un}\,\Pda dv_g\geqslant C^{-1}>0.
\]
Moreover, \eqref{lim:iinttoinfty} holds.
We define
\begin{align*}
\label{def:f(t)}
f(t):=(\calP(\{1\}))^{-1}\int_{[0,1]}\alpha e^{\alpha t}\,\Pda
\end{align*}
and
\begin{equation*}
\label{def:kappan}
\kappan:=\frac{\lambdan\calP(\{1\})}{\iint_{[0,1]\times\Om}e^{\alpha\un}\,\Pda dx}.
\end{equation*}
With such definitions, $\un$ satisfies \eqref{eq:Ye}.
In view of \eqref{lim:iinttoinfty}, we have  $\kappan\to0$.
Consequently, $r\equiv0$
and furthermore
\begin{equation*}
\label{eq:kftomasses}
\kappan f(\un)\stackrel{*}{\rightharpoonup}\sum_{i=1}^nn_i\delta_{p_i}
\end{equation*}
weakly in the sense of measures.
We are left to establish that $n_i=8\pi$, for all $i=1,\ldots,m$ and
that the blow-up points satisfy condition~\eqref{eq:blowuppointpositions}.
\par
To this end, we adapt some ideas of \cite{ye}.
We set
\begin{equation*}
\label{def:Nerif}
F(t)=(\calP(\{1\}))^{-1}\int_{[0,1]}e^{\alpha t}\,\Pda.
\end{equation*}
Then, $F'(t)=f(t)$ and furthermore we have the following.
\par
\textit{Claim~A.} As $t\to+\infty$, we have:
\begin{align}
\label{f(t)F(t)asymptotics}
&f(t)=e^t+o(e^t)
&&F(t)=e^t+o(e^t).
\end{align}
Proof of Claim~A.
Let
\begin{align}
\label{eq:poet}
&p(t)=\int_{[0,1)}\alpha e^{\alpha t}\,\calP(d\alpha),
&&P(t)=\int_{[0,1)}e^{\alpha t}\,\calP(d\alpha).
\end{align}
Then, $f(t)=e^t+\tau^{-1}p(t)$, $F(t)=e^t+\tau^{-1}P(t)$.
\par
For any given $\ee>0,$ we fix $0<\delta_\ee\ll1$ such that 
\begin{equation*}
\int_{[1-\delta_\ee,1)} e^{-(1-\alpha)t}\calP(d\alpha)\leqslant\calP( [1-\delta_\ee,1) )<\frac\ee2.
\end{equation*}
Correspondingly, we take $t_\ee\gg1$ such that 
\begin{equation*}
\begin{split}
\int_{[0,1-\delta_\ee)} e^{-(1-\alpha)t}\calP(d\alpha)\leqslant e^{-\delta_\ee t}<\frac \ee 2\qquad \forall t\geqslant t_\ee.
\end{split}\end{equation*}
It follows that $e^{-t}P(t)=\int_{[0,1)} e^{- (1-\alpha) t} \calP(d\alpha)<\ee$ whenever $t\geqslant t_\ee$.
That is, $P(t)=o(e^t)$, and the second  part of \eqref{f(t)F(t)asymptotics} is established.
The first part of \eqref{f(t)F(t)asymptotics} follows by observing that $0\leqslant p(t)\leqslant P(t)$. Hence Claim~A is established.
\par
\textit{Claim~B}.
We have
\begin{equation}
\label{eq:kFtomasses}
\kappan F(\un)\stackrel{*}{\rightharpoonup}\sum_{i=1}^nn_i\delta_{p_i}.
\end{equation} 
Proof of Claim~B.
Let $p(t)$ be the function defined in \eqref{eq:poet}.
For $p_i\in\mathcal S$, let $B_\rho(p_i)$ be such that $\overline{B_\rho(p_i)}\cap\mathcal S=\{p_i\}$.
Let $\varphi\in C_c(B_\rho(p_i))$.
Let $\ee>0$ and $t_\ee'\gg1$ be such that $e^{-t}p(t) <\ee/2$ whenever  $t\geqslant t_\ee'.$
We have
\begin{align*}
\left |  \kappan\int_\Om p(\un)\varphi
\right|\leqslant &\kappan\int_{\un\geqslant t_\ee'}p(\un) |\varphi| +\kappan\int_{\un<t_\ee'}p(\un)|\varphi|\\
\leqslant&\frac\ee2\int_{\Om}\kappan e^{\un}|\varphi|+\kappan\max_{[0,t_\ee']}p\int_\Om|\varphi|< c \ee
\end{align*}
for sufficiently large $n$.
Since $\ee$ and $\varphi$ are arbitrary, we conclude that $\kappan p(\un)\stackrel{*}{\rightharpoonup}0$
weakly in the sense of measures.
By the same argument, we conclude that $\kappan P(\un)\stackrel{*}{\rightharpoonup}0$
weakly in the sense of measures.
Therefore,
\[
\kappan\int_\Om e^{\un}\varphi=\kappan\int_\Om (f(\un)-p(\un))\varphi\to n_i\varphi(p_i)
\]
and
\[
\kappan\int_\Om F(\un)\varphi=\kappan\int_\Om(e^{\un}+P(\un))\varphi\to n_i\varphi(p_i).
\]
Hence, \eqref{eq:kFtomasses} is established.
\par
\textit{Claim~C}: There holds 
\begin{equation}
\label{eq:massquant}
n_i=8\pi
\end{equation}
for all $i=1,\ldots,m$.
\par
Proof of Claim~C.
We adapt the complex analysis argument in \cite{ye}.
For the sake of simplicity, throughout this proof we omit the index $n$.
We fix a blow-up point $p\in\mathcal S$ and without loss of generality
we assume that $p=0$ and $\xi(0)=0$.
We define 
\[
W(t)=\kappa F(t)+c\,t
\]
and we consider the Newtonian potential $N=(4\pi)^{-1}\log(z\bz)$ so that $\Delta N=\delta_0$.
We define 
\begin{align*}
&H=\frac{u_z^2}2,
&& K=N_z\ast\{e^\xi\chi_{B_\rho}[W(u)]_z\}.
\end{align*}
It is readily checked that the function $S=H+K$ satisfies $\partial_{\bz}S=0$
in $B_\rho$.
It follows that $S$ converges uniformly to a holomorphic function $S_0$.
On the other hand, we have $\un\to u_0$ in $W^{1,q}(B_\rho)$, $q\in[1,2)$, where
\[
u_0(x)=\frac{n_p}{4\pi}\log(z\bz)+\om,
\]
where $\om$ is smooth in $B_\rho$.
Taking limits for $H$ we thus find that $H\to H_0$, where
\[
H_0=\frac{n_p^2}{32\pi^2z^2}-\frac{n_p}{4\pi z}\om_z+\frac12\om_z^2.
\]
On the other hand,
we may write
\[
K=N_{zz}\ast\{e^\xi\chi_{B_\rho}W(u)\}-N_z\ast\{[e^\xi\chi_{B_\rho}]_zW(u)\}.
\]
In view of Claim~B we have $W(u)\stackrel{*}{\rightharpoonup}n_p\delta_0+c_0u_0$.
Recalling that $N_z=(4\pi z)^{-1}$, $N_{zz}-(4\pi z^2)^{-1}$,
we thus compute
\[
N_{zz}\ast\{e^\xi\chi_{B_\rho}W(u)\}\to-\frac{n_p}{4\pi z^2}-N_z\ast\{e^\xi\chi_{B_\rho}c_0u_{0,z}\}
\]
and 
\[
N_z\ast\{[e^\xi\chi_{B_\rho}]_zW(u)\}\to
\frac{n_p}{4\pi z}\xi_z(p)+c_0N_z\ast\{e^\xi\chi_{B_\rho}u_0\},
\]
pointwise in $B_\rho\setminus\{0\}$.
Therefore, $K\to K_0$ where
\[
K_0=-\frac{n_p}{4\pi z^2}-N_z\ast\{e^\xi\chi_{B_\rho}c_0u_{0,z}\}
-\frac{n_p}{4\pi z}\xi_z(p)-c_0N_z\ast\{e^\xi\chi_{B_\rho}u_0\}.
\]
Since $S_0=H_0+K_0$ is holomorphic, by balancing singularities we derive
\begin{align*}
&\frac{n_p^2}{32\pi^2}=\frac{n_p}{4\pi},
&&\frac{n_p}{4\pi}\om_z(p)=-\frac{n_p}{4\pi}\xi_z(p).
\end{align*}
Hence, \eqref{eq:massquant} holds and Claim~C is established.
Moreover, we have
\[
\om_z(p)=-\xi_z(p).
\]
Finally, observing that 
\[
\om(x)=b(p)+\frac{1}{4}\sum_{p'\neq p}\log\frac{1}{d_g(p',p)},
\]
we derive \eqref{eq:blowuppointpositions}
\end{proof}

For later application, we now explicitely state the mass quantization
for Neri's equation on a domain with Dirichlet boundary conditions and 
on a compact Riemannian 2-manifold without boundary.
Let $\Om\subset\Rd$ be a smooth bounded domain and let $G_\Om$ be the Green's function defined by 
\begin{equation*}
\label{def:GreenOmega}
\left\{
\begin{split}
-\Delta_x G_\Om(x,y)=&\delta_y &&\textrm{ in\ }\Om \\
 G_\Om(\cdot,y) =&0&&\mbox{on }\de \Om.
\end{split}
\right.
\end{equation*}
It is well known that
\begin{equation*}
\label{eq:Greendomain}
G_\Om(x,y)=\frac{1}{2\pi}\log\frac{1}{|x-y|}+h(x,y),
\end{equation*}
where $h(x,y)$ is the regular part of $G_\Om$.
Assuming that $\calP$ satisfies \eqref{calH},
problem~\eqref{eq:Neri} takes the form:
\begin{equation}
\label{eq:NeriDirichlet}
\left\{\begin{split}
- \Delta u =&\lambda\,\int_{[0,1]}\frac{\alpha e^{\alpha u}\,\calP (d\alpha)}
{\iint_{[0,1]\times\Om} e^{\alpha u}\,\calP (d\alpha)dx}&&\text{in }\Omega \\
 u=&0&&\text{on }\de\Om
\end{split}\right.
\end{equation}

By the maximum principle, we have $u>0$ in $\Om$.
In the next lemma, we exclude the existence of blow-up on $\partial\Om$.
\begin{lemma}
\label{lem:nobdryblowup}
Let $(\lambdan,\un)$ be a solution sequence to \eqref{eq:NeriDirichlet}
with $\lambda\to\lambda_0$.
There exists a tubular neighborhood $\Om_\delta$ of $\partial\Om$
and a constant $C>0$ such that $\|\un\|_{L^\infty(\Om_\delta)}\leqslant C$.
\end{lemma}
\begin{proof}
By a result in \cite{GNN}, p.~223, it is known that there exists a tubular neighborhood $\Om_\delta$ of $\de\Om$,
depending on the geometry of $\Om$ only,
such that any solution to a problem of the form $-\Delta u=f(u)$
satisfying $u=0$ on $\partial\Om$, where $f(t)\geqslant0$ is Lipschitz continuous, 
has no 
stationary points in $\Om_\delta$.
 We may assume that $\partial \Om_\delta \cap \Om\cap \mathcal S = \emptyset.$ Let $x_n\in\bar{\Om}_\delta$ be such that $u_n(x_n) = \max_{\bar{\Om}_\delta}u_n.$ Arguing by contradiction, suppose that  $\un(x_n)\to+\infty$.  Since $u_n=0$ on $\partial \Om$, and since $u_n $ is uniformly bounded on $\partial \Om_\delta\cap \Om$, then, for $n$ sufficiently large, $x_n\in \Omega_\delta$ and $\nabla u_n(x_n)=0$, a contradiction.
\end{proof}
At this point, the following result readily follows.
\begin{proposition}[Mass quantization for the Dirichlet problem]
\label{prop:Dirichletmassquantization}
Assume \eqref{calH}.
Let $(\lambdan,\un)$ be a solution sequence to the problem~\eqref{eq:NeriDirichlet}
with $\lambda=\lambdan\to\lambda_0$.
Then, up to subsequences, exactly one of the following alternatives holds:
\begin{enumerate}
  \item[(i)] There exists a solution $u_0$ to equation~\eqref{eq:NeriDirichlet}
 with $\lambda=\lambda_0$ such that $\un\to u_0$;
 \item[(ii)] There exists a finite set $\mathcal S=\{p_1,\ldots,p_m\}\subset\Om$
such that $\un\to u_0$ in $W_0^{1,q}(\Om)$ for all $q\in[1,2)$, where
\begin{equation*}
\label{def:uoDirichlet}
u_0(x)=8\pi\sum_{i=1}^mG_\Om(x,p_i).
\end{equation*}
Moreover, the points $p_i$ satisfy the condition $\nabla R_i(p_i)=0$,
where
\[
R_i(x)=h(x,p_i)+\sum_{j\neq i}G_\Om(x,p_j)
\]
and
\[
\lambdan\,\int_{[0,1]}\frac{\alpha e^{\alpha\un}\,\calP (d\alpha)}
{\iint_{[0,1]\times\Om} e^{\alpha\un}\,\calP (d\alpha)dx}\,dx
\stackrel{*}{\rightharpoonup}8\pi\sum_{j=1}^m\delta_{p_j}(dx),
\]
weakly in the sense of measures.
\end{enumerate}
\end{proposition}
We note that Proposition~\ref{prop:Dirichletmassquantization} is consistent with Theorem~1 in \cite{MW}.
\begin{proof}[Proof of Proposition~\ref{prop:Dirichletmassquantization}]
The proof is a direct consequence of Proposition~\ref{prop:localblowup}
and Proposition~\ref{prop:reductiontoYe}.
In view of Lemma~\ref{lem:nobdryblowup}, blow-up does not occur on the boundary $\partial\Om$.
Since $u>0$, alternative (ii) in Proposition~\ref{prop:localblowup}
cannot occur.
Moreover, at a given blow-up point $p_i\in\mathcal S$,
we have
\[
b(p_i)+\frac{1}{4}\sum_{j\neq i}\log\frac{1}{d_g(p_i,p_j)}
=h(x,p_i)+\sum_{j\neq i}G_\Om(x,p_j)
\]
and since $g$ is Euclidean, $\xi\equiv0$.
\end{proof}
Similarly, let $(M,g)$ be a compact orientable Riemannian surface without boundary.
Let $G_M$ be the Green's function defined by
\begin{equation*}
\label{def:Greenmanifold}
\left\{
\begin{split}
-\Delta_g G_M(x,y)=&\delta_y -\frac{1}{|M|}\\
\int_M G_M(x, y)dv_g =&0.
\end{split}
\right.
\end{equation*}
Then, 
\[
G_M(x,y)=\frac 1{2\pi}\log\frac{1}{d_g (x,y)}+h(x,y),
\]
where $h$ is the regular part of $G_M$, see \cite{Aub}.
Assuming \eqref{calH}, Neri's equation on a manifold \eqref{EL0}
takes the form
\begin{equation}
\label{eq:Nerimanifold}
\left\{
\begin{split}
-\Delta_g v=& \lambda\,\int_{[0,1]}\frac{\alpha (e^{\alpha v}- \frac{1}{|M|} \int_M e^{\alpha v}\,dv_g)\,}
{\iint_{[0,1]\times M} e^{\alpha v}\,\calP (d\alpha)\,dv_g}\calP (d\alpha)&&\text{ in }M \\
\int_{M }v\,dv_g=&0,
\end{split}
\right.   
\end{equation}
The following holds.
\begin{proposition}[Mass quantization for the problem on $M$]
\label{prop:manifoldmassquantization}
Assume \eqref{calH}.
Let $(\lambdan,\vn)$ be a solution sequence to the problem~\eqref{eq:Nerimanifold}
with $\lambda=\lambdan\to\lambda_0$.
Then, up to subsequences, exactly one of the following alternatives holds:
\begin{enumerate}
  \item[(i)] There exists a solution $v_0$ to equation~\eqref{eq:Nerimanifold}
 with $\lambda=\lambda_0$ such that $\vn\to v_0$;
 \item[(ii)] There exist a finite number of points $p_1,\ldots,p_m\in M$
such that $\vn\to v_0$ in $W^{1,q}(M)$ for all $q\in[1,2)$, where
\begin{equation*}
\label{def:voDirichlet}
v_0(x)=8\pi\sum_{j=1}^mG_M(x,p_j).
\end{equation*}
Moreover, the points $p_i$, $i=1,\ldots, m$ satisfy the condition $\nabla R_i(p_i)=-\nabla\xi(p_i)$,
where $g=e^{\xi(x)}(dx_1^2+dx_2^2)$
and
\[
R_i(x)=h(x,p_i)+\sum_{j\neq i}G_M(x,p_j).
\]
Furthermore,
\[
\lambdan\,\int_{[0,1]}\frac{\alpha e^{\alpha\vn}\,\calP (d\alpha)}
{\iint_{[0,1]\times M} e^{\alpha\vn}\,\calP (d\alpha)dv_g}\,dx
\stackrel{*}{\rightharpoonup}8\pi\sum_{j=1}^m\delta_{p_j},
\]
weakly in the sense of measures.
\end{enumerate}
\end{proposition}
\begin{proof}
The proof is analogous to the proof of Proposition~\ref{prop:Dirichletmassquantization}.
\end{proof}
\section{An improved Moser-Trudinger inequality}
\label{sect-mt}
We derive an improved Moser-Trudinger inequality for the functional
\eqref{eq:Nerifunctional} defined on a bounded domain $\Om\subset\R^2$
which will be needed in the proof of Theorem~\ref{main-link}.
\par
We recall that
the classical Moser-Trudinger sharp inequality  \cite{Mo} states that
\begin{equation}
\label{fontana}
C_{MT}:=\sup\left\{\int _\Om e^{4\pi u^2}:\ u\in H^1_0(\Om),\ \|\nabla u \|_2\leqslant 1\right\}<+\infty,
\end{equation}
where the constant $4\pi$ is best possible. 
Moreover, the embeddings $u\in H_0^1(\Om)\to e^u\in L^1(\Om)$
and $v\in H^1(M)\to e^v\in L^1(M)$
are compact.
For a proof, see, e.g., Theorem~2.46 pag. 63 in \cite{Aub}.
\par
In view of the elementary inequality
\begin{equation*}
| u |\leqslant \frac{\|\nabla u \|_2^2}{16 \pi}+4\pi  \frac{u^2}{\|\nabla u\|_2^2}
\end{equation*}
we derive using \eqref{fontana} that
\begin{equation*}
\int_{\Om}  e^{|u|}\,dx\leqslant C_{MT}e^{ \frac1{16\pi}\|\nabla u\|_2^2},\qquad \forall u\in H^{1}_0(\Om).
\end{equation*}
In particular, the standard Moser-Trudinger functional
\begin{equation*}
\label{MT-ifunct}
I_\lambda(u)=\frac{1}{2}\|\nabla u\|_2^2-\lambda\log\int_{\Om}e^{u}\,dx
\end{equation*}
is bounded below for all $\lambda\leqslant 8\pi$ and 
\[
\inf_{u\in H^1_0(\Om)} I_\lambda (u)=-\infty
\]
whenever $\lambda >8\pi$. 
From the arguments above it can be shown that if $\mathrm{supp}\calP\cap\{-1,1\}\neq\emptyset$,
then Neri's functional \eqref{eq:Nerifunctional} is also bounded below 
on $H_0^1(\Om)$ if and only if $\lambda\leqslant 8\pi$, and that 
\[
\inf_{u\in H^1_0(\Om)} J_\lambda (u)=-\infty
\]
whenever $\lambda >8\pi$ and  $\mathrm{supp}\calP\cap\{-1,1\}\neq\emptyset$. 
More precisely, we have
\begin{prop}
\label{prop:TM}
Let $\mathrm{supp}\calP\cap\{-1,1\}\neq\emptyset$.
Then, the functional $J_\lambda(u)$ is bounded from below on $H^1_0(\Om),$ if and only if $\lambda \leqslant8\pi.$
\end{prop}
Proposition~\ref{prop:TM} was established for functions $u\in H^1(M)$
satisfying $\int_Mu=0$, where $M$ is a two-dimensional Riemannian manifold, in \cite{RZ}.
The proof for $u\in H^1_0(\Om)$ is similar. For the sake of completeness, we outline it below. 
In the improved Moser-Trudinger inequality we show that the best constant
in Proposition~\ref{prop:TM} may be lowered if the ``mass" of $u$
is suitably distributed. 
Namely, following ideas of \cite{Aub, cl}, we prove:
\begin{prop}[Improved Moser-Trudinger inequality]
\label{mtmigliorata} 
Let   $d_0>0$ and  $a_0\in (0,1/2).$ Then, for any $\ee>0,$ there exists a constant $K=K(\ee,d_0,a_0)>0$ such that if $u\in H_0^1(\Om)$ satisfies 
\begin{equation}
\label{omegai}
\frac{\iint_{I\times\Om_i } e^{\alpha u}\calP (d\alpha)dx}{\iint_{I\times\Om} e^{\alpha u}\calP (d\alpha)dx}\geqslant a_0,\qquad i=1,2
\end{equation}
where  $\Om_1, \Om_2\subset \Om $ are two measurable sets verifying dist$(\Om_1,\Om_2)\geqslant d_0$, then it holds
\begin{equation}\label{tesimt}
\iint_{I\times\Om} e^{\alpha u}\calP (d\alpha)dx\leqslant K  \exp\left\{ \left( \frac1{32 \pi} +\ee\right) \|\nabla u\|^2_2\right\}.
\end{equation}
\end{prop}
We begin by outlining the proof of Proposition~\ref{prop:TM}.
\begin{proof}[Proof of Proposition~\ref{prop:TM}]
The ``if" part is immediate and was already used in \cite{ne}.
Indeed, since
\begin{equation*}
\iint_{I\times\Om}  e^{\alpha u}\calP(d\alpha) dx  \leqslant \int_\Om e^{|u|} dx\leqslant C_{MT} e^{
\frac1{16\pi} \|\nabla u\|_2^2},
\end{equation*}
for all  $u \in H^{1}_0(\Om)$. Therefore $J_\lambda$ is bounded below
if $\lambda\leqslant 8\pi$.
On the other hand
the value $8\pi$ is also optimal,
provided that $\mathrm{supp}\calP\cap\{-1,1\}\neq\emptyset$.
Indeed, the following holds:
We need only prove that
\begin{equation}
\label{opt}
\inf_{u\in H^1_0(\Om),}J_\lambda(u)=-\infty, \qquad \quad\forall \lambda >8\pi.
\end{equation}
Assume that  $1\in $ supp$\mathcal{P}$ (the case $-1\in $ supp$\mathcal{P}$ is similar). Since the functional $I_\lambda(u)$
is unbounded below for $\lambda >8\pi$, then also the functional
\begin{equation*}
I_\lambda(u)_{|_{u\geqslant 0}}=\frac{1}{2}\|\nabla u\|_2^2-\lambda\log\int_{\Om}e^{u}\,dx, \qquad u\in H^1_0(\Om),\,\, u\geqslant 0
\end{equation*}
is unbounded below for $\lambda >8\pi$. At this point we observe that for every $0<\delta<1$ and $u\geqslant 0$, $u\in H^1_0(\Om),$ we have:
\begin{equation}
\label{MT-posit}
\begin{split}
J_\lambda (u) &= \frac 12 \|\nabla u\|_{2}^2-\lambda \log \iint_{I\times\Om} e^{\alpha u}\calP (d\alpha )dx
\leqslant \frac 12 \|\nabla u\|_{2}^2-\lambda \log \iint_{[1-\delta,1]\times \Om}  e^{\alpha u}\calP(d\alpha)  dx\\
&\leqslant \frac 12\|\nabla u\|_{2}^2-\lambda \log \int_\Om  e^{(1-\delta) u} dx -\lambda \log (\calP([1-\delta,1]))\\
&= \frac{1}{(1-\delta)^2} \left[\frac 12 \|(1-\delta)\nabla  u\|_2^2-\lambda (1-\delta)^2 \log\left(\int_\Om  e^{(1-\delta) u} dx\right) \right] -\lambda \log (\calP([1-\delta,1])) \\
&= \frac{1}{(1-\delta)^2} I_{\lambda (1-\delta)^2} \left ((1-\delta) u\right) -\lambda \log (\calP([1-\delta,1])). 
\end{split}
\end{equation}
Hence, for $\lambda (1-\delta)^2>8\pi$, the right hand side of last inequality is unbounded from below  and so
\[
\inf_{u\in H^1_0 (\Om)}J_\lambda (u) =-\infty  \qquad \mbox{ for any }\lambda >\frac{8\pi}{(1-\delta)^2}.
\]
Since $\delta\in(0,1)$ is arbitrary, \eqref{opt} follows.
\end{proof}
In order to prove Proposition~\ref{mtmigliorata},
We adapt  some ideas contained in \cite{cl}, Proposition~1.   
\begin{proof}
Let $g_1$ and $g_2$ be smooth functions defined on $\Om$
such that $0\leqslant g_i\leqslant 1,$  $i=1,2$;
$g_i \equiv 1$ on $\Om_i$, $i=1,2$;  $\mathrm{supp}g_1\cap\mathrm{supp}g_2 =\emptyset$;
$|\nabla g_i |\leqslant c(d_0),$  $i=1,2$. 
We may assume that $\|g_1 \nabla u\|_{L^2(\Om)}\leqslant \| g_2\nabla  u\|_{L^2(\Om)}$  (otherwise it is sufficient to switch the functions $g_1$ and $g_2$). 
Denote, for every real number $t$,  $t^+= \max \{0,t\}$ and let $a>0$.
In view of the elementary inequality
\begin{equation*}
g_1(| u|-a)^+\leqslant \frac1{16 \pi}  \|\nabla \left[ g_1(| u|-a)^+\right] \|^2_{L^2(\Om)}
+\frac{4\pi (g_1 (| u|-a)^+)^2}{ \|\nabla \left[ g_1(| u|-a)^+\right] \|^2_{L^2(\Om)}},
\end{equation*}
we derive from \eqref{fontana} 
that
\begin{equation}
\label{passo}
\int_{ \Om}e^{g_1(| u|-a)^+}\leqslant  C_{MT}\exp\left\{\frac1{16 \pi}  \|\nabla \left[ g_1(| u|-a)^+\right] \|^2_{L^2(\Om)}\right \}.
\end{equation}
Hence, using \eqref{omegai} and  \eqref{passo}, 
and using the elementary  inequality $(A+B)^2\leqslant (1+\tau)A^2+c(\tau) B^2$
for any $\tau>0 $, we have
\begin{align*}
\label{mtmigl}
\iint_{I\times\Om}&e^{\alpha u} \Pda dx\leqslant \frac{e^a}{a_0} \iint_{I\times \Om_1}e^{(\alpha u-a)^+} \mathcal{P}(d\alpha )dx\\
&\leqslant \frac{e^a}{a_0} \iint_{I\times \Om}e^{g_1(\alpha u-a)^+} \mathcal{P}(d\alpha) dx\leqslant \frac{e^a}{a_0} \int_{\Om}e^{g_1(| u|-a)^+}dx \\
&\leqslant\frac{ C }{a_0 }  \exp \left\{  \frac1{16 \pi}  \|\nabla\left[  g_1(|u|-a)^+ \right]\|^2_{L^2(\Om)}+a\right\}\\
& \leqslant\frac{ C }{a_0 } \exp \left\{  \frac1{16 \pi} \left[ (1+\tau) \|g_1\nabla u \|^2_{L^2(\Om)}  \right.\right.\\
&\left. \left.    +c(\tau) \| (| u|-a)^+\nabla g_1\| _{L^2(\Om)}^2\right]
+a\right\}\\
& \leqslant\frac{ C }{a_0 }  \exp \left\{  \frac1{32 \pi} \left[ (1+\tau) \|(g_1+g_2)\nabla u \|^2_{L^2(\Om)}+c(\tau,d_0) \| (| u|-a)^+\| _{L^2(\Om)}^2\right]+a\right\}\\
&\leqslant\frac{ C }{a_0} \exp \left\{  \frac1{32 \pi} \left[ (1+\tau) \|\nabla u \|^2_{L^2(\Om)}+c(\tau,d_0) \| (| u|-a)^+\| _{L^2(\Om)}^2\right]+a\right\}
\end{align*}
for some small $\tau>0,$ where $C=C(\Om)$. 
For a given real number $\eta\in(0,|\Om|)$, let 
\[
a=a(\eta, u) =\sup \left\{ c\geqslant 0 : \mbox{ meas }\{x\in \Om : | u(x)|\geqslant c\} >\eta\right\}.
\]
We have
\begin{equation*}\label{unsigm}
\begin{split}
 \|(| u|-a)^+\|^2_{L^2(\Om)} =\int_{\{x\in \Om :|u|>a\}} (|u|-a )^2 \leqslant\eta^{\frac12}C  \|\nabla u \|^2_2
\end{split}
\end{equation*}
Using the Schwarz and Poincar\'e inequalities, we finally derive
\begin{equation*}
a\eta\leqslant\int_{ \{ | u| \geqslant a\} }|u| \leqslant  |\Om|^\frac12 \left( \int_\Om |u|^2\right)^\frac12 \leqslant C \|\nabla u\|_2,
\end{equation*}
and therefore, for any small $\delta>0$,
\begin{equation*}\label{av1}
a \leqslant \frac \delta 2\|\nabla u \|_2^2 +\frac{C^2}{2\delta \eta^2}.
\end{equation*}
The asserted improved Moser-Trudinger inequality \eqref{tesimt} is completely established.
\end{proof}
\section{Struwe's Monotonicity Trick: a unified form}
\label{sec:struwe}
The aim of this section is to establish Struwe's Monotonicity Trick
in a unified form convenient for application to both Theorem~\ref{main-link}
and Theorem~\ref{main-mp}.
Let $\Lambda\subset \R_+$ be a bounded interval and let $\mathcal H$ be a Hilbert  space. 
In this section, for $\lambda\in\Lambda$, we consider functionals of the form 
\[
\mathcal{J}_\lambda (w)= \frac{1}{2}\left\| w\right\|
^{2}-\lambda \mathcal{G}(w)
\] 
defined for every $w\in \mathcal H$,  where  $\mathcal G\in \mathcal C^2 (\mathcal H;\R)$ satisfies:
\begin{align}
\label{calG}
&\mathcal G'\ \mathrm{is\ compact},  
&&\left\langle \mathcal{G''}(w)\fip,\fip\right \rangle\geqslant0\ \mathrm{for\ every}\ w,\fip\in\mathcal H.
\end{align}
We do not make any sign assumption on $\mathcal G$.
Let $V\subset \R ^m$ be a bounded domain.
We consider the family
\begin{equation*}
\mathcal F_\lambda := \left\{  f\in \mathcal C (V;\mathcal H ): f \mbox{ satisfies }  \mathscr P(\de V)\right\}
\end{equation*}
where $\mathscr P(\de V)$ is a set of properties defined on $\de V$, including 
a property of the form:
\begin{equation}\label{limsup}
\limsup_{\theta \rightarrow \de V} \mathcal J _\lambda (f(\theta)) \leqslant A,  
\end{equation}
for some  $A\in [-\infty,+\infty)$. 
We assume that  for every $ \lambda,\lambda' \in \Lambda$ it holds that $\mathcal F_\lambda \not=\emptyset$ and
\begin{equation*}
 \lambda'<\lambda\Longrightarrow \mathcal F_{\lambda'} \subseteq \mathcal F_{\lambda}.
\end{equation*}
Under these assumptions, we define the minimax value: 
\begin{equation*}\label{minmax-c}
c_\lambda=\inf_{f  \in\mathcal F_\lambda} \sup_{\theta \in V } \mathcal{J}_\lambda (f(\theta)).
\end{equation*}
and we assume that  $c_\lambda$ is finite for every $\lambda.$ 
Since for every fixed $w\in \mathcal H$ the function $\lambda^{-1}\mathcal J_\lambda$ is non-increasing
with respect to $\lambda$, 
$\lambda^{-1}c_\lambda$ is non-increasing  as well. 
Therefore, writing $c_\lambda=\lambda(\lambda^{-1}c_\lambda)$ we see that
\[
\left. c'_\lambda=\frac{d c_{\lambda+\ee}}{d\ee}\right|_{\ee=0}
\]
is well-defined and finite for almost every $\lambda\in \Lambda$.
We shall use the monotonicity trick in the following form.  
\begin{prop}[Struwe's Monotonicity Trick]
\label{struwe}
Suppose that $\mathcal{G}$ satisfies assumptions~\eqref{calG}
and let $\lambda\in \Lambda$ be such that $c'_\lambda$ exists. 
If 
\begin{equation}
\label{-A}
c_\lambda>A,
\end{equation}
then $c_\lambda$ is a critical value for $\mathcal J_\lambda$.
That is, there exists $w\in \mathcal H $ such that $\mathcal J_\lambda (w)=c_\lambda$ and $\mathcal J'_\lambda (w) =0$.
\end{prop}
In order to prove Proposition~\ref{struwe} we need some lemmas.
\begin{lem}
\label{lemma-uno}
Let $\lambda\in \Lambda$ be such that $c'_\lambda$ exists. Then, there exist two  constants $K=K(c'_\lambda)>0$ and $\bar \ee>0$  such that for all $\ee\in (0,\bar \ee)$ and  for all $w\in \mathcal H $ satisfying
\begin{equation}
\label{1}
\mathcal J_\lambda (w)\geqslant c_\lambda -\ee \qquad \mbox{ and }\qquad  \mathcal J_{\lambda-\ee} (w)\leqslant c_{\lambda -\ee}+\ee 
 \end{equation}
it holds that 
\begin{description}
\item{(i)} $\mathcal{G} (w) \leqslant K$;
\item{(ii)} $\| w\| \leqslant \sqrt{2(c_\lambda +\lambda  K+1)}$ ;
\item{(iii)} $| \mathcal{J}_{\lambda}(w) -c_\lambda | \leqslant ( \frac{c_\lambda}\lambda+|c_{\lambda}'|+2) \ee.$
\end{description}
\end{lem}
\begin{proof}   
Proof of (i).
By  \eqref{1}  it follows that
\begin{equation*}
\begin{split}
 \mathcal{G}(w)&= \frac{\mathcal{J}_{\lambda-\ee} (w)-\mathcal{J}_{\lambda }(w)}{\ee} 
\leqslant \frac 1\ee \left( c_{\lambda -\ee} - c_{\lambda}\right)+2
=2-c'_\lambda +o(1)
\end{split}
\end{equation*}
for $\ee$ sufficiently small. Hence, $(i)$ follows with $K= 3-c'_\lambda.$
\par
Proof of (ii). By the monotonicity property of $\lambda^{-1}\mathcal J_\lambda$, 
we have $\mathcal J_\lambda (w) \leqslant \frac{\lambda }{\lambda-\ee}\mathcal J_{\lambda-\ee}(w) $  
for all $w\in \mathcal H$. Consequently, in view of  \eqref{1} and (i) we have
\begin{equation*}
\begin{split}
\frac 12\|w\|^2&=  \mathcal{J}_\lambda (w)+ \lambda \mathcal{G}(w) 
\leqslant    \frac{\lambda }{\lambda-\ee}     (c_{\lambda-\ee} +\ee) +\lambda K\\
&\leqslant  \left( 1+ \frac{\ee }{\lambda-\ee}\right)  ( c_\lambda -\ee c'_\lambda+ o(\ee)  +\ee  )+ \lambda K 
\leqslant     c_\lambda+ \lambda K +o(1)
\end{split}
\end{equation*}
for $\ee$ sufficiently small, so that  the (ii) follows. 
\par
Proof of (iii). Similarly,  
\begin{equation*}
\begin{split}
c_\lambda -\ee \leqslant  \mathcal{J}_\lambda (w) \leqslant  \frac{\lambda }{\lambda-\ee}    (c_{\lambda-\ee} +\ee) 
&=\left( 1+ \frac{\ee }{\lambda-\ee}\right)  ( c_\lambda -\ee c'_\lambda+ o(\ee)  +\ee  ) \\
&= c_\lambda +\ee\left (\frac{c_\lambda  }{\lambda-\ee} -c'_\lambda  +1 +o(1)\right).
\end{split}
\end{equation*}
for $\ee$ sufficiently small, and (iii) follows.
\end{proof}
In the next lemma we show the existence of bounded Palais-Smale sequences for $\mathcal J_\lambda$ 
at the level $c_\lambda.$
\begin{lem}
\label{lemma-due} 
Under the assumptions of Proposition~\ref{struwe},  for every $\ee>0$ sufficiently small there exists $w_\ee\in \mathcal H$ such that
\begin{description}
\item{(i)}  $|\mathcal J _\lambda (w_\ee) -c_\lambda| \leqslant C\ee$ 
\item{(ii)} $\|w_\ee\| \leqslant C,$
\item{(iii)} $\|\mathcal J_\lambda' (w_\ee )\|_{\mathcal H'}\rightarrow 0 $ as $\ee\rightarrow 0.$
\end{description}
\end{lem}
\begin{proof}
Proof of (i)--(ii).
Let $\ee>0.$ By definition of $c_{\lambda-\ee}$
and by the infimum property, there exists $f_\ee\in\mathcal F_{\lambda-\ee}$ such that
\begin{equation*}
\label{suppp}
\sup_{\theta \in V}\mathcal J_{\lambda-\ee} (f_\ee(\theta)) \leqslant c_{\lambda-\ee}+\ee.
\end{equation*}
Moreover, since $ f_\ee \in \mathcal F_{\lambda-\ee}\subseteq  \mathcal F_{\lambda}$,  
by definition of $c_\lambda$ and by the supremum property, there exists $\theta_\ee \in V$ such that
\[
\mathcal J_{\lambda} (f_\ee(\theta_\ee)) \geqslant c_{\lambda}-\ee.
\]
We conclude that $w_\ee:=f_\ee(\theta_\ee) \in\mathcal H$ satisfies \eqref{1}.
Consequently, Lemma~\ref{lemma-uno} implies that,  for $\ee$ sufficiently small,
$\| w_\ee\| \leqslant C$ and $|\mathcal{J}_{\lambda}(w_\ee)-c_\lambda | \leqslant \ee C$,
for some $C>0$ independent of $\ee$.
\par
Proof of (iii).
For $\delta>0$ we set
\[
N_\delta:=\{ w \in \mathcal H : \|w\| \leqslant C, |\mathcal J_\lambda (w) -c_\lambda |<\delta \},
\]
and we note that in view of the already established properties (i)--(ii) we have
$N_\delta\not=\emptyset$. 

Suppose that the claim is false. We shall derive a contradiction by constructing a suitable deformation.
To this end, we assume that there exists $\delta>0$ such that  $\|\mathcal J'_\lambda (w)\|\geqslant \delta$ for every $w \in N_\delta.$ 
Let $\xi\in \mathcal C(\R,\R)$ be a cut-off function such that $0\leqslant \xi\leqslant 1$,
$\xi(t)=0$ if and only if $t \leqslant -2$, $\xi(t)=1$ if and only if $t\geqslant -1$.
We set
\begin{equation*}
\xi_\ee (\theta )  := \xi \left( \frac{\mathcal J_{\lambda} (f_\ee (\theta)) - c_{\lambda} }{\ee} \right),\qquad \forall \theta \in V.
\end{equation*}
For all $\theta \in V$ we define the deformation
\begin{equation*}\label{7} 
\tilde f_\ee(\theta)=f_\ee(\theta)-\sqrt{\ee} \,\,\xi_\ee (\theta ) \frac{\mathcal J'_{\lambda} ( f_\ee (\theta))}{\|\mathcal J'_{\lambda} ( f_\ee (\theta))\|}.
\end{equation*}
In view of \eqref{limsup} and \eqref{-A}, we have $\xi_\ee (\theta )=0$ when $\ee $ is sufficiently small 
and $\theta$ is near $\de V$. In particular, $\tilde f_\ee$ satisfies $\mathscr P(\de V)$ and therefore $\tilde f _\ee \in \mathcal F_{\lambda -\ee}.$
By Taylor expansion of $\mathcal J_\lambda$ and in view of assumption~\eqref{calG} we have for all $w,\varphi\in \mathcal H$
\begin{equation*}
\begin{split}
\mathcal J_{\lambda }(w+\varphi ) &\leqslant \mathcal J_{\lambda }(w) +\mathcal  J'_{\lambda }(w)\varphi + \frac 12 \| \varphi \| ^{2} .
\end{split}
\end{equation*}
Consequently, 
\begin{equation}\label{cancelletto}
\begin{split}
\mathcal J_{\lambda}(\tilde f_\ee(\theta ))&\leqslant\mathcal J_{\lambda }(f_\ee(\theta ))
-\sqrt{\ee} \,\,\xi_\ee (\theta ) \frac{ \|\mathcal J'_{\lambda}(f_\ee(\theta) )\|^2}{\|\mathcal J'_{\lambda}(f_\ee(\theta) )\|}+\frac\ee 2\xi_\ee^2(\theta).
\end{split}
\end{equation}
Now we claim that 
\begin{equation}\label{star}
\sup_{\theta \in V} \mathcal J_\lambda (\tilde f_\ee(\theta)) \leqslant c_\lambda- \frac\ee 2
\end{equation}
for all sufficiently small $\ee>0.$ We prove \eqref{star} by considering two cases.
\par
\underline{Case I}: $\xi_\ee(\theta) <1.$
In this case $\mathcal J_\lambda (f_\ee(\theta)) <c_\lambda -\ee$ and hence,  using \eqref{cancelletto} 
 we estimate
\begin{equation*}\label{cancelletto1}
\begin{split}
\mathcal J_{\lambda }(\tilde f _\ee(\theta ))  &\leqslant c_\lambda -  \frac \ee 2 .
\end{split}
\end{equation*}
\par
\underline{Case II}: $\xi_\ee(\theta) =1.$
In this case, we have that $f_\ee(\theta)$ 
satisfies the assumptions of Lemma \ref{lemma-uno} and consequently $f_\ee(\theta)\in N_\delta$ 
for all $\ee>0$ sufficiently small. By the contradiction assumption we then have $\|\mathcal J_\lambda'(f_\ee(\theta))\|\geqslant \delta$. 
It follows that, using again \eqref{cancelletto} and Lemma \ref{lemma-uno},
\begin{equation*}\label{cancelletto2}
\begin{split}
\mathcal J_{\lambda }(\tilde f_\ee(\theta ))&\leqslant\mathcal J_{\lambda }(f_\ee(\theta))-\sqrt{\ee}\delta+\frac\ee2
\leqslant c_{\lambda}+\ee\left(\frac{c_\lambda}{\lambda}+|c_\lambda'|+\frac52\right)-\sqrt{\ee}\delta
\leqslant c_{\lambda}- \frac{\sqrt{\ee}\delta}{2}
\end{split}
\end{equation*}
for all $\ee>0$ sufficiently small. Taking $\ee\leqslant\delta^2$ we conclude that 
$\mathcal J_{\lambda}(\tilde f_\ee(\theta ))\leqslant c_\lambda-\frac\ee 2$ and claim \eqref{star} is established.
But then we derive
\[
c_\lambda = \inf_{f  \in\mathcal F_\lambda} \sup_{\theta \in V } \mathcal{J}_\lambda (f(\theta))\leqslant  c_\lambda-\frac\ee2,
\]
a contradiction. 
\end{proof}
At this point the proof of Proposition~\ref{struwe} follows by standard compactness arguments.
\begin{proof}[Proof of Proposition~\ref{struwe}] 
Let $w_\en\in\mathcal H$ be a bounded Palais-Smale sequence as obtained in Lemma~\ref{lemma-due}. 
Then there exists $w_0\in\mathcal H$ such that $w_\en\rightharpoonup w_0$ weakly in $\mathcal H$. 
On the other hand, we have 
\begin{align*}
\langle\mathcal J_\lambda'(w_\en),w_\en-w_0\rangle=& \|w_\en-w_0\|^2+\langle w_0,w_\en-w_0\rangle \\
&-\lambda\langle\mathcal G '(w_\en)-\mathcal G '(w_0),w_\en-w_0\rangle-\lambda\langle\mathcal G '(w_0),w_\en-w_0\rangle. 
\end{align*}
Hence, by the compactness assumption \eqref{calG} on  $\mathcal{G}'$ and using the fact that $\|\mathcal J'_\lambda (w_\en)\| \rightarrow 0$,  we obtain
\[
o(1)\|w_\en-w_0\| =\langle \mathcal J_\lambda'(w_\en), w_\en-w_0\rangle=\|w_\en-w_0\|^2 +o(1)
\]
so that $w_\en\rightarrow w_0$ strongly in $\mathcal H.$   This implies, by the continuity of $\mathcal J_\lambda$ on $\mathcal H$,
$\mathcal J_\lambda (w_\en)\rightarrow \mathcal J_\lambda (w_0) =c_\lambda$.
Moreover, since $\mathcal{G}'(w_\en)\varphi \rightarrow \mathcal{G}'(w_0)\varphi$ for all $\varphi\in\mathcal H$,
we conclude that 
$o(1)=\mathcal{J'}_\lambda (w_\en)\varphi=\langle w_\en,\varphi\rangle-\lambda\mathcal{G}'(w_\en) \varphi \rightarrow 
\mathcal{J'}_\lambda (w_0)\varphi$
and therefore $\mathcal{J'}_\lambda (w_0)=0$.
\end{proof}
\section{Proof of Theorem~\ref{main-link}}
\label{sec:minimax}
We begin by establishing the existence almost everywhere of solutions.
We note that assumption~\eqref{calH} is not necessary for this part of Theorem~\ref{main-link}.
We shall use the monotonicity trick, as established in Proposition~\ref{struwe},
 to construct the minimax critical values  for Neri's functional~\eqref{eq:Nerifunctional}
defined on a non-simply connected domain $\Om\subset\Rd$.
Namely, we consider the functional
\begin{equation*}
J_{\lambda}(u)=\frac{1}{2}\int_{\Om }\left\vert \nabla u\right\vert^{2}\,dx 
-\lambda\log\left( \iint_{I\times \Om }e^{\alpha u}\,\mathcal{P}(d\alpha)dx\right),
\end{equation*}
for all $u\in H_0^1(\Om)$.
Our aim in this subsection is to show the following.
\begin{prop}
\label{prop:criticalvalueonOmega}
For almost every $\lambda\in(8\pi,16\pi)$, there exists a saddle-type critical value $c_\lambda>-\infty$ for $J_\lambda$. 
\end{prop}
The construction of  $c_\lambda$ relies on an idea
originally introduced by \cite{DJLW} for the standard mean field equation \eqref{eq:standardmfe}.
Such an idea was also exploited in \cite{COS} for the Toda system. Here,
we extend the argument to the case of mean field equations including a probability measure.
\par
For every $u\in H^1_0(\Om)$ we consider the measure:
\begin{align*}
\mu_u=\frac{\int_{I}e^{\alpha u}\calP (d\alpha)}{\iint_{I\times\Om}e^{\alpha u} \calP (d\alpha)dx}\,dx
\in\mathcal M(\Om)
\end{align*}
and the corresponding ``baricenter" of $\Om$:
\begin{align*}
m(u)=&\frac{\iint_{I\times\Om} xe^{\alpha u} \calP (d\alpha)dx}{\iint_{I\times\Om} e^{\alpha u}\calP (d\alpha)dx}
=\int_\Om x\,d\mu_u\in\Rd.
\end{align*}
We note that $\mu_u(\Om)=1$ and $|m(u)|\leqslant\sup\{|x|:\ x\in\Om\}$.
In the following lemma we  show that, as a consequence of Proposition~\ref{mtmigliorata},
if the functional $J_\lambda$ given by \eqref{eq:Nerifunctional} is unbounded
below along a sequence $\un\in H_0^1(\Om)$, and if $\lambda\in(8\pi,16\pi)$,
then $\un$ blows up at exactly one point $x_0\in\overline\Om$.
\begin{lem}
\label{baricentri}
Let $\lambda \in (8\pi,16\pi)$ and  $\{u_n\}\subset H^1_0(\Om)$ 
be a sequence such that $J_\lambda(u_n)\rightarrow -\infty$.
Then, there exists $x_0\in\overline\Omega$ such that
\begin{equation*}
\mu_{\un}\rightharpoonup \delta_{x_0}\quad \mbox{weakly* in }\mathcal C(\bar \Om)'
\qquad \textrm{and\ }\qquad m(\un)\to x_0.
\end{equation*}
\end{lem}
\begin{proof}
Throughout this proof, for simplicity, we denote $\mu_n=\mu_{\un}$.
For every fixed $r>0$ we denote by $\mathcal Q_n(r)$ the concentration function of $\mu_n$,
namely,
\[
\mathcal Q_n(r)= \sup_{x\in \Om} \int_{B(x,r)\cap \Om} \mu_n,\qquad (r>0).
\]
For every $n$ we take $\tilde x_n\in\overline\Om$ such that
$\mathcal Q_n(r/2) =\int_{B(\tilde x_n, r/2)\cap \Om} \mu_n$.
Upon taking a subsequence, we may assume that $\tilde x_n\to x_0\in\overline\Om$.
We set $\Om_1^n =B(\tilde x_n, r/2)\cap \Om$ and $\Om_2^n = \Om \setminus B(\tilde x_n,r)$ 
and we note that $\mathrm{dist}(\Om_1^n, \Om_2^n)\geqslant r/2$.  
Since $J_\lambda (u_n) \rightarrow -\infty $ and $\lambda < 16\pi$, in view of Proposition~\ref{mtmigliorata} we 
conclude that $\min\{\mu_n(\Om_1^n), \mu_n(\Om_2^n)\}\rightarrow0$.
In particular, $\min\{\mathcal Q_n (r/2),1-\mathcal Q_n(r)\}\leqslant\min\left( \mu_n(\Om_1^n), \mu_n(\Om_2^n)\right)\rightarrow0.$
On the other hand, for every fixed $r>0$ let $k_r\in\mathbb N$ be such that $\Om$ is covered
by $k_r$ balls of radius $r/2$. Then, $1=\mu_n(\Om)\leqslant k_rQ_n (r/2)$, so that $Q_n (r/2)\geqslant k_r^{-1}$
for every $n$. We conclude that necessarily $\mathcal Q_n(r)\to1$ 
as $n\to\infty$. Since $r>0$ is arbitrary, we derive in turn that
$1-\mathcal Q_n(r/2)=\mu_n(\Om\setminus B(\tilde x_n, r/2))\to0$
as $n\to\infty$. That is, $\mu_{\un}\rightharpoonup\delta_{x_0}$.
It follows that $m(\un)=\int_\Om x\,d\mu_n\to x_0$.
\end{proof}
Let $\Gamma_1\subset\Om$ be a non-contractible curve 
which exists in view of the non-simply connectedness assumption on $\Om$. 
Let $\mathbb D =\{(r,\theta):0\leqslant r<1,\ 0\leqslant \theta <2\pi\}$ be the unit disc. 
We now define the sets of functions which will be used in the minimax argument
by setting
\begin{equation*}
\mathcal D_\lambda= \left\{
\begin{split}
h\in \mathcal C(\mathbb D,H_0^{1,2}(\Om))\ s.t.:\
&\textrm{(i)}\quad\lim_{r\rightarrow 1} \sup_{\theta \in [0,2\pi)} J_\lambda (h(r,\theta))=-\infty,\\
&\textrm{(ii)}\quad m(h(r,\theta))\mbox{\ can be extended continuously to } \bar {\mathbb D},\\
&\textrm{(iii)}\quad m (h(1,\cdot)):\partial\mathbb D\rightarrow\Gamma_1\mbox{ has degree 1 }\\
\end{split}
\right\}.
\end{equation*}
\begin{lem}
\label{lemma-dom-1}
For any $\lambda \in (8\pi,16\pi)$ the set $\mathcal D_\lambda$ is non-empty.
\end{lem}
\begin{proof}
We assume that $1\in $ supp $\mathcal{P}$ (the case $-1 \in$ supp $\mathcal{P}$ can be treated in the same way).  Let $\gamma_1(\theta):[0,2\pi) \rightarrow \Gamma_1$ be a parametrization of $\Gamma_1$ and let 
$\ee_0>0$ be sufficiently small so that $B(\gamma_1(\theta),\ee_0)\subset \Om$. 
Let $\varphi_\theta(x)=\ee_0^{-1}(x-\gamma_1(\theta))$ so that $\varphi_\theta(B(\gamma_1(\theta),\ee_0))=B(0,1)$. 
We define ``truncated Green's function'':
\begin{equation*}
V_{r}(X)=\left\{
\begin{split}
&4\log\frac1{1-r} \qquad \mbox{ for }X\in B(0,1-r)\\
&4\log\frac1{|X|}\qquad \mbox{ for }X\in B(0,1)\setminus B(0,1-r)\\
\end{split}\right.
\end{equation*}
and
\begin{equation*}
v_{r,\theta}(x)=\left\{
\begin{split}
&0\qquad \mbox{ for }x\in \Om \setminus B(\gamma_1(\theta),\ee_0))\\
&V_{r}(\varphi_\theta(x))\quad \mbox{ for }x\in  B(\gamma_1(\theta),\ee_0).\\
\end{split}\right.
\end{equation*}
We set 
\begin{equation}
\label{def:h}
h(r,\theta)(x)=v_{r,\theta}(x), x\in\Om. 
\end{equation}
\par
\textit{Claim:} The function $h$ defined in \eqref{def:h} satisfies $h\in \mathcal D_\lambda$.
\par
We check (i).  As in \eqref{MT-posit}, we note that for  any $\delta\in (0,1)$ and for any $(r,\theta)$ it holds
\begin{equation*}
J_\lambda (h(r,\theta))\leqslant\frac12\|\nabla h(r,\theta)\|^2_2
-\lambda\log\int_\Om e^{(1-\delta)h(r,\theta)}-\lambda \log (\mathcal{P}([1-\delta,1]))
\end{equation*}
We have $\|\nabla h(r,\theta)\|^2_2=-32\pi\log(1-r)$
and
\begin{eqnarray*}
\int_\Om e^{(1-\delta)h(r,\theta)}\,dx&&=\ee_0^2\int_{B(0,1)}e^{(1-\delta ) V_{r} (X)}dX+|\Om|-\pi \ee_0^2\\
&&=2\pi\ee_0^2\left( \int_0^{1-r}e^{(1-\delta) 4 \log  \frac{1}{1-r}} \rho d\rho + \int_{1-r}^1 e^{(1-\delta) 4 \log  \frac{1}{\rho}} \rho d\rho \right)+O(1)\\
&&=\pi\ee_0^2\left[  \left(\frac{1}{ 1-r}\right)^{2-4\delta} + \frac{1}{1-2\delta} \left(\left(\frac{1}{ 1-r}\right)^{2-4\delta}  -1\right)\right]+O(1),
\end{eqnarray*}
where $O(1)$ is bounded independently of $(r,\theta)$.
It follows that if $0<\delta \leqslant 1/4$, then
\begin{equation*}
\log\int_\Om e^{(1-\delta )h} = (2-4\delta) \log  \frac{1}{1-r}+ O(1).
\end{equation*}
We conclude that
\[
J_\lambda(h) \leqslant 2\left( 8\pi -\lambda (1-2\delta) \right)\log \frac{1}{1-r}-\lambda\log \mathcal{P}([1-\delta,1])+O(1).
\]
Since $\lambda >8\pi, $ by choosing $\delta>0$ sufficiently small we conclude that $h(r,\theta)$ defined by \eqref{def:h} 
satisfies property~(i).
\par
We check (ii)--(iii).
To this aim it is sufficient to prove that $\lim_{r\rightarrow1}m(h(r,\theta))=\gamma_1(\theta)$
uniformly with respect to $\theta\in [0,2\pi)$. 
We consider again the measures 
\[
\mu_{r,\theta}=\frac{\int_I e^{\alpha h(r,\theta)}\mathcal{P}(d\alpha)}
{\iint_{I\times \Om} e^{\alpha h(r,\theta)}\mathcal{P}(d\alpha)}.
\]
We claim that
for every $\ee\in (0,\ee_0)$, we have  
$\lim_{r\rightarrow 1}\int_{B(\gamma_1(\theta),\ee)}\mu_{r,\theta}=1$ uniformly with respect to $\theta \in [0,2\pi)$.
Indeed, fix $\ee\in (0,\ee_0),$ and let $\delta=\delta(\ee)\in(0,1)$ such that $B(0,\delta)\subset\varphi_\theta( B(\gamma_1(\theta),\ee))$. 
We write
\begin{equation}
\label{4}
\int_{B(\gamma_1(\theta), \ee)}\mu_{r,\theta} =\frac{I+II}{I+III}
\end{equation}
where
\[
I=\iint_{I\times \varphi^{-1}_\theta\left( B(0, \delta)\right)}e^{\alpha v_{r,\theta}(x)}dx 
\]
and
\[
II=\iint_{I\times (B(\gamma_1(\theta),\ee)\setminus \varphi^{-1}_\theta \left( B(0, \delta)\right)}e^{\alpha v_{r,\theta}(x)}\,dx,
\qquad III=\iint_{I\times (\Om \setminus \varphi^{-1}_\theta \left( B(0, \delta)\right)}e^{\alpha v_{r,\theta}(x)}\,dx.
\]
Since $1\in $ supp$\mathcal{P}$, for every  $\bar\alpha \in (0,1]$, supp$\mathcal{P}\cap[\bar \alpha,1]\not= \emptyset$. 
Then, for any $r>1-\delta$ we have
\begin{align*}
I&\geqslant \int_{\bar\alpha} ^1 \int_{ \varphi^{-1}_\theta\left( B(0, \delta)\right)}e^{\alpha v_{r,\theta}(x)}\,dx\\
&\geqslant\ee_0^2\mathcal{P}([\bar\alpha,1]) \int_{B(0,\delta)} e^{\bar\alpha V_r (X)}dX\\
&=\ee_0^2\mathcal{P}([\bar\alpha,1]) \left[  \int_{B(0,1-r)} \frac{1}{(1-r)^{4\bar\alpha}} dX + \int_{B(0,\delta)\setminus B(0,1-r) }  \frac{1}{|X|^{4\bar\alpha}} dX\right]\\
&=\pi\ee_0^2\mathcal{P}([\bar\alpha,1])\left[(1-r)^{2-4\bar\alpha } + \frac{1}{2\bar \alpha -1}\left((1-r)^{2-4\bar\alpha}  -\delta^{2-4\bar\alpha }\right)\right] \\
\end{align*}
Hence, choosing  $\bar\alpha>1/2$, we derive $I\rightarrow+\infty$ as $r\rightarrow 1$
uniformly with respect to $\theta$. Moreover, 
\begin{equation*}
\label{B2}
0\leqslant II\leqslant III
\leqslant\pi\ee_0^2\left(\frac1{\delta^2} -1\right)+|\Om|,
\end{equation*}
Letting $r\to1$ in \eqref{4} we conclude that $\int_{B(\gamma_1(\theta),\ee)}\mu_{r,\theta}\to1$
for any $\ee>0$, uniformly with respect to $\theta$, and consequently $\mu_{h(r,\theta)}\stackrel{*}{\rightharpoonup}\delta_{\gamma(\theta)}$.
In turn, we derive $\lim_{r\rightarrow1}m(h(r,\theta))=\gamma_1(\theta)$, and therefore $h$
satisfies properties (ii)--(iii) in the definition of $\mathcal D_\lambda$.
We conclude that $h\in\mathcal D_\lambda$.
\end{proof}
We define the minimax value:
\[
c_\lambda= \inf_{h\in \mathcal D_\lambda} \sup_{(r,\theta)\in \mathbb D} J_\lambda(h(r,\theta)).
\]
In view of Lemma~\ref{lemma-dom-1}, we have $c_\lambda<+\infty$.
The following property relies on the nontrivial topology of $\Om$ in an essential way.
\begin{lem}
\label{lemma-dom-2}
For any $\lambda\in (8\pi,16\pi)$, $c_\lambda >-\infty.$
\end{lem}
\begin{proof}
Denote by $B$ a bounded component of $\Rd\setminus \Om$ such that $\Gamma_1$ encloses $B$.
By the continuity and degree properties defining 
$\mathcal D_\lambda$, we have $m(h(\mathbb D))\supset B$ for all $h\in\mathcal D_\lambda$.
Arguing by contradiction, we assume  $c_\lambda=-\infty$. Then, there exists a sequence
$\{h_n\}\subset\mathcal D_\lambda$ such that $\sup_{(r,\theta)\in \mathbb D} J_\lambda(h_n(r,\theta))\rightarrow -\infty.$ 
Let $x_0$ be an interior point of $B$.  
For every $n$ we take $(r_n,\theta_n)\in \mathbb D$ such that $m(h_n(r_n,\theta_n))=x_0$.
In view of Lemma~\ref{baricentri}, there exists $\tilde x_0\in\overline\Om$ such that 
$m(h_n(r_n,\theta_n))\to\tilde x_0$. But then $x_0=\tilde x_0\in \stackrel{o}{B}\cap\overline\Om=\emptyset$,
a contradiction.
\end{proof}
\begin{lem}\label{lemma-dom-3}
For $8\pi<\lambda_1 \leqslant \lambda_2<16\pi$,  we have $\mathcal D_{\lambda_1}\subseteq \mathcal D_{\lambda_2}$.
\end{lem}
\begin{proof}
It is sufficient to note that whenever  $J (u)\leqslant 0$ it is $\log \iint_{I\times \Om} e^{\alpha u} \geqslant 0,$ which implies that 
\[ 
 J_{\lambda_1}(u)\geqslant J_{\lambda_2}(u)  \qquad \mbox{ for } \quad  8\pi<\lambda_1 <\lambda_2<16\pi.
\]
Hence, $\mathcal D_{\lambda_1}\subseteq \mathcal D_{\lambda_2}$ for every $8\pi<\lambda_1 <\lambda_2<16\pi$.
\end{proof}
We set
\begin{equation}
\label{g0}
G(u)=  \log\left( \iint_{I\times \Om }e^{\alpha u}\mathcal{P}(d\alpha )\right)
\end{equation}
so that Neri's functional~\eqref{eq:Nerifunctional} takes the form
\[
J_\lambda (u)= \frac{1}{2}\int_{\Om }\left\vert \nabla u\right\vert
^{2}-\lambda G(u).
\] 
\begin{lem}
\label{proprg-uno}
The function $G:H_0^1(\Om)\rightarrow \R$ defined by \eqref{g0} satisfies assumptions~\eqref{calG}.
\end{lem}
\begin{proof}
For any $u,\phi,\psi\in H_0^1(\Om)$ we have:
\begin{equation*}
G^{\prime}(u)\phi=\iint_{I\times\Om}\frac{\alpha\phi e^{\alpha u}}{\iint_{I\times \Om }e^{\alpha u}},
\end{equation*}
and therefore the compactness of $G'$ follows by compactness of the Moser-Trudinger embedding 
as stated in Section \ref{sect-mt}. Moreover
\begin{equation*}
\left\langle G^{\prime\prime }(u)\phi ,\psi \right\rangle
=\frac{\left( \iint_{I\times \Om }\alpha ^{2} \phi \psi e^{\alpha u}\right) \left(\iint_{I\times \Om
}e^{\alpha u}\right)-\left(\iint_{I\times \Om}\alpha \phi  e^{\alpha u}\right)\left(\iint_{I\times\Om}\alpha\psi e^{\alpha u}\right)}
{\left(\iint_{I\times\Om} e^{\alpha u}\right)^2},
\end{equation*}
so that we obtain $\left\langle G^{\prime\prime}(u)\phi,\phi\right\rangle\geqslant0$ using the Schwarz inequality.
\end{proof}
Now we are able to prove Proposition~\ref{prop:criticalvalueonOmega}.
\begin{proof}[Proof of Proposition~\ref{prop:criticalvalueonOmega}]
In view of Lemma~\ref{proprg-uno}, Lemma~\ref{lemma-dom-1}, Lemma~\ref{lemma-dom-2} and Lemma~\ref{lemma-dom-3},
we may apply Proposition~\ref{struwe}
with $\mathcal H=H_0^1(\Om)$, $\mathcal G(u)=G(u)$,
$V=\mathbb D$, $A=-\infty$ and  $\mathcal F_\lambda=\mathcal D_\lambda$. 
\end{proof}
Now we are ready to prove Theorem~\ref{main-link}.
\begin{proof}[Proof of Theorem~\ref{main-link}] 
By Proposition~\ref{prop:criticalvalueonOmega},  for almost every $\lambda\in(8\pi,16\pi)$, 
the value $c_\lambda$ is a critical value for $J_\lambda$, 
which is achieved by a critical point $u\in H_0^1(\Om)$. 
Hence, for almost every $\lambda\in(8\pi,16\pi)$
we obtain a solution to \eqref{eq:Neri}.
Now we assume that $\calP$ satisfies \eqref{calH}.
Let $\lambda_0\in(8\pi,16\pi)$.
Using the first part of  Theorem~\ref{main-link}, 
there exists a solution sequence $(\lambda_n,\un)$ to \eqref{eq:Neri}
such that $\lambdan\in(8\pi,16\pi)$ and
$\lambda_{n}\rightarrow\lambda_0$.
In view of Proposition~\ref{prop:Dirichletmassquantization},
blow-up cannot occur for $\lambda_0\in\left(8\pi,16 \pi\right)$.
Therefore, $\un$ converges uniformly to a solution
$u_0$ to \eqref{eq:Neri} with $\lambda=\lambda_0$.
\end{proof}
\section{Proof of Theorem~\ref{main-mp}}
\label{sec:mp}
Let $\mathcal E=\{v\in H^1(M):\ \int_Mv\,dv_g=0\}$.
The variational functional for Neri's equation~\eqref{EL0} defined on a manifold $M$
is given by
\begin{equation*}
\label{Jlambda}
J_{\lambda }(v)=\frac{1}{2}\int_{M }\left\vert \nabla_g v\right\vert
^{2}dv_g -\lambda \log\left( \frac{1}{|M|} \iint_{I\times M }e^{\alpha v}\,\mathcal{P}(d\alpha )dv_g\right),
\end{equation*}
where $v\in\mathcal E$.
We begin by establishing the following.
\begin{prop}
\label{critic2}
For almost every $\lambda\in\left(8\pi,\frac{\mu _{1}(M )\left\vert M \right\vert}{\int_{I}\alpha ^{2}\mathcal{P}(d\alpha) } \right)$, 
there exists a mountain pass critical value $c_\lambda>0$ for $J_\lambda$.
\end{prop}
It is convenient to set
\begin{equation*}
\label{g}
G(v)=  \log\left( \frac{1}{|M|} \iint_{I\times M }e^{\alpha v}\mathcal{P}(d\alpha)dv_g\right)
\end{equation*}
so that $J_\lambda (v)= \frac{1}{2}\int_{M }\left\vert \nabla v\right\vert^{2}-\lambda G(v)$.
Henceforth, throughout this subsection, for simplicity we denote $\nabla=\nabla_g$, $\Delta=\Delta_g,$
$dx=dv_g$, and we omit the integration measure when it is clear from the context.
Then,
\begin{equation*}
G^{\prime}(v)\phi=\iint_{I\times M }\frac{\alpha e^{\alpha v}\phi }{\iint_{I\times M }e^{\alpha v}}
\end{equation*}
and
\begin{equation*}
\label{g2}
\left\langle G^{\prime\prime}(v)\phi,\psi\right\rangle
=\frac{\left(\iint_{I\times M}\alpha^{2}e^{\alpha v}\phi\psi\right)
\left(\iint_{I\times M}e^{\alpha v}\right) - \left( \iint_{I\times M} \alpha e^{\alpha v}\phi\right)
\left(\iint_{I\times M }\alpha e^{\alpha v}\psi\right)}{\left(\iint_{I\times M } e^{\alpha v}\right)^2 }.
\end{equation*}
In particular, 
$G(0)=0$; in view of Jensen's inequality we have $G(v)\geqslant0$ for every $v\in \mathcal E$; 
$G'(0)=0$ and $G'$ is compact in view of Section \ref{sect-mt}.
Furthermore, the Schwarz inequality implies that
$\langle G''(v)\varphi,\varphi\rangle\geqslant0$ for all $\varphi\in\mathcal E$,
and we compute
\begin{equation}
\label{Gsecond(0)}
\|G''(0)\|:=\inf_{\phi\in\mathcal E\setminus\{0\}}
\frac{\langle G''(0)\phi,\phi\rangle}{\|\nabla\phi\|_2^2}
=\frac{\int_I\alpha^2\,\mathcal P(d\alpha)}{\mu_1(M)|M|},
\end{equation}
where $\mu_1(M)$ is the first nonzero eigenvualue defined in \eqref{mu1}.
\begin{lemma}
\label{mount-pass-struct}
Let $\mathcal{P}$ satisfy assumption~\eqref{assumpt:suppP1}
and suppose $8\pi<\frac{\mu _{1}(M)\left\vert M\right\vert}{\int_{I}\alpha^{2}\mathcal{P}(d\alpha)}$. 
If $\lambda<\frac{\mu _{1}(M)\left\vert M\right\vert}{\int_{I}\alpha^{2}\mathcal{P}(d\alpha)}$,
then $v\equiv0$ is a strict local minimum for $J_\lambda$.
Moreover, if $\lambda>8\pi$, then there exists $v_1\in\mathcal E$ such that $\|v_1\|\geqslant1$ and $J_\lambda(v_1)<0$.
In particular,  $J_{\lambda }$ has a mountain-pass geometry for each $\lambda\in \left(8\pi,
\frac{\mu _{1}(M)\left\vert M\right\vert}{\int_{I}\alpha^{2}\mathcal{P}(d\alpha)}\right)$.
\end{lemma}
\begin{proof}
By Taylor expansion and by properties of $G$, we have
\begin{equation}
\label{JTaylorat0}
J_\lambda(\phi)\geqslant\frac12\left(1-\lambda\|G''(0)\|\right)\|\phi\|^2+o(\|\phi\|^2)
\end{equation}
Therefore, in view of \eqref{Gsecond(0)}, $v\equiv0$ is a strict local minimum for
$J_\lambda$ whenever $\lambda<\|G''(0)\|^{-1}=\frac{\mu _{1}(M)\left\vert M\right\vert}{\int_{I}\alpha^{2}\mathcal{P}(d\alpha)}$.
We recall from \cite{RZ} that $\int_Me^{\alpha v}$ is increasing with respect to $\alpha$ for all $v\in\mathcal E$.
Indeed,
\[
\frac{d}{d\alpha}\int_Me^{\alpha v}=\int_Mve^{\alpha v}
=\int_{v\geqslant0}ve^{\alpha v}-\int_{v<0}|v|e^{\alpha v}\geqslant0.
\]
Consequently, similarly as in \eqref{MT-posit}, we compute
\begin{align*}
J_\lambda(v)\leqslant\frac1{(1-\delta)^2}\left[\frac12\int_M|\nabla(1-\delta)v|^2
-(1-\delta)^2\lambda\log\int_Me^{(1-\delta)v}\right]-\lambda\log\calP([1-\delta,1]).
\end{align*}
We fix $p_0\in  M$ and $r_0>0$ a constant smaller than the injectivity
radius of $M$ at $p_0$.  For every $\ee>0$, we define
\begin{equation*}
{v}_{\epsilon }(p)=\left\{
\begin{array}{ccc}
\log \frac{\epsilon ^{2}}{(\epsilon ^{2}+ d_g (p,p_0)^2 )^{2}}
& \text{in }\mathcal{B}_{r_{0}}&  \\
\log \frac{\epsilon ^{2}}{(\epsilon ^{2}+r_{0}^{2})^{2}} & \text{otherwise,}
&
\end{array}
\right.
\end{equation*}
where $\mathcal B_{r_0} = \{ p \in M : d_g (p, p_0)  < r_0\}$  denotes the geodesic
ball of radius $r_0$ centered at $p_0$. 
We set
\begin{equation*}
\tilde v_{\epsilon }(x):={v}_{\epsilon }(x)-\frac{1}{\left\vert M
\right\vert }\int_{M }{v}_{\epsilon } dv_g.
\end{equation*}
so that $\tilde v_\ee \in \mathcal E$.
A standard computation  yields
\begin{equation*}\label{v_ee}
\int_M |\nabla_g\tilde v_\ee |^2dv_g= 32 \pi \log\frac1\ee+O(1),\\
\end{equation*}
and
\begin{equation*}
\log\int_Me^{\tilde v_\ee}=2\log\frac1\ee+O(1)
\end{equation*}
where $O(1)$ is independent of $\ee$.
See, e.g.,  \cite{R} for the details.
Then,
\[
J_\lambda\left(\frac{\tilde v_\ee}{1-\delta}\right)
\leqslant\frac2{(1-\delta)^2}\left[8\pi 
-(1-\delta)^2\lambda\right]\log\frac1\ee+O(1).
\]
It follows for any $\lambda>8\pi$ there exists $0<\delta\ll1$
such that $J_\lambda\left(\tilde v_\ee/(1-\delta)\right)\to-\infty$ as $\ee\to0$.
Since $\|\tilde v_\ee\|\to+\infty$ as $\ee\to0$, taking $v_1=v_\ee$
with $\ee$ sufficiently small, we conclude the proof.
\end{proof}
We note that if $\lambda$
belongs to a compact subset of $\left(8\pi,
\frac{\mu _{1}(M)\left\vert M\right\vert}{\int_{I}\alpha^{2}\mathcal{P}(d\alpha)}\right)$, the function $v_1$ may be chosen independently of $\lambda$. 
We denote by 
$\Gamma$ the set of paths
\[
\Gamma= \{\gamma\in \mathcal C([0,1], \mathcal E) : \gamma(0)=0, \gamma(1)=v_1\}.
\]
We set
\[
c_\lambda=\inf_{\gamma\in \Gamma} \sup_{t\in [0,1]} J_\lambda (\gamma(t)).
\]
Note  that the value $c_\lambda$ is finite and non negative. In particular, in the following lemma we prove that  for  
$\lambda \in \left( 8\pi,\frac{\mu _{1}(M)\left\vert M\right\vert}{\int_{I}\alpha^{2}\mathcal{P}(d\alpha)}\right),$  $c_\lambda$ is 
strictly positive.
\begin{lem}
\label{cposit2}
Let  $\lambda \in \left( 8\pi,\frac{\mu _{1}(M)\left\vert M\right\vert}{\int_{I}\alpha^{2}\mathcal{P}(d\alpha)} \right).$  For every $\ee >0, $ there exists $\rho_\ee >0$ (independent on $\lambda$), such that
\[
c_\lambda\geqslant\frac{\rho^2_\ee}{2}\left(1-\lambda\,\frac{\int_{I}\alpha^{2}\mathcal{P}(d\alpha)}{\mu _{1}(M)\left\vert M\right\vert}-\ee\right).
\]

\end{lem}

\begin{proof}
In view of the Taylor expansion \eqref{JTaylorat0}, $\ee>0$ there exists a constant $\rho_\ee \in (0,1)$ independent of $\lambda$, 
such that for any $v\in \mathcal E$ satisfying $\|v\| \leqslant \rho_\ee$ 
we have
\[
J_\lambda(v)\geqslant \frac{1}2\left(1-{\lambda \,\frac{\int_{I}\alpha^{2}\mathcal{P}(d\alpha)}{\mu _{1}(M)\left\vert M\right\vert} }  
- \ee\right)\|v\|^2.
\]
In particular, for any $v\in \mathcal E$ such that $\|v\|=\rho_\ee$ we have
\[
J_\lambda(v)\geqslant \frac{\rho_\ee^2}2\left(  1-  {\lambda \|G''(0)\|}- \ee\right).
\]
By continuity of $\gamma\in \Gamma $, we derive in turn that
\[
\sup_{t\in[0,1]}J_\lambda(\gamma(t))\geqslant\frac{\rho_\ee^2}2\left(1- {\lambda\,\frac{\int_{I}\alpha^{2}\mathcal{P}(d\alpha)}{\mu _{1}(M)\left\vert M\right\vert}} - \ee\right),
\]
and asserted strict positivity of $c_\lambda$ follows.
\end{proof}

\begin{proof}[Proof of Proposition~\ref{critic2}]
We use the
Struwe's Monotonicity Trick as stated in  Proposition~\ref{struwe} with $\mathcal H=\mathcal E,$ $V=[0,1],$ $\mathcal F_\lambda =\Gamma$ and $A=0.$ 
\end{proof}
Now we are ready to prove Theorem~\ref{main-mp}.
\begin{proof}[Proof of Theorem \ref{main-mp}] 
Let  $\lambda\in\left(8\pi,\frac{\mu _{1}(M)\left\vert M\right\vert}{\int_{I}\alpha^{2}\mathcal{P}(d\alpha)}\right)$ 
be such that $c'_\lambda$ exists.  
By  Proposition~\ref{critic2},  we have that $c_\lambda$ is a critical value  for $J_\lambda$, 
which is achieved by a critical point $v\in \mathcal E$. Such a $v$ is a solution to problem \eqref{EL0}. This proves the first part of Theorem~\ref{main-mp}.
To prove the second part of  Theorem~\ref{main-mp}, assume that $\calP$ satisfies \eqref{calH}. Let  $\lambda_0\in \left(8\pi,\frac{\mu _{1}(M)\left\vert M\right\vert}{\int_{I}\alpha^{2}\mathcal{P}(d\alpha)}\right)$. 
By the first part of Theorem~\ref{main-mp} there exists
a solution sequence $(\lambdan,v_n)$ to \eqref{EL0}
with $\lambda=\lambdan$,
such that $\lambdan\to\lambda_0$.
In view of the mass quantization, as stated in Proposition~\ref{prop:manifoldmassquantization},
blow-up cannot occur.
Therefore, $\vn$ converges uniformly to a solution $v_0$
for problem~\eqref{EL0} with $\lambda=\lambda_0$.
\end{proof}
\section*{Acknowledgements}
We thank Professor Takashi Suzuki for introducing us to mean field equations
motivated by turbulence, for many interesting discussions and for pointing out reference \cite{ye}.
This research was carried out within the Marie Curie program FP7-MC-IRSES-247486
\textit{Mathematical studies on critical non-equilibrium phenomena via mean field theories
and theories on nonlinear partial differential equations.}


\begin{thebibliography}{99}
\bibitem{Aub} 
T.~Aubin, 
\textit{Some Nonlinear Problems in Riemannian Geometry}. 
Springer Monographs in Mathematics,  Springer-Verlag, Berlin, 1998.
\bibitem{BBH}
F.~Bethuel, H.~Brezis, F.~H\'{e}lein, 
\textit{Ginzburg-Landau vortices.} 
Progress in Nonlinear Differential Equations and their Applications, 13. 
Birkh\"{a}user Boston, Inc., Boston, MA, 1994.
\bibitem{bm} 
H.~Brezis, F.~Merle, 
Uniform estimates and blow-up behavior for solutions of $-\Delta u=V(x)e^{u}$ in dimension 2,
Comm.\ Partial Differential Equations 16 (1991), 1223--1253.
\bibitem{BouchetVenaille}
F.~Bouchet, A.~Venaille,
Statistical mechanics of two-dimensional and geophysical flows,
Phys.\ Rep.\ 515 (2012), 227--295.
\bibitem{CLMP}
E.~Caglioti, P.L.~Lions, C.~Marchioro, M.~Pulvirenti,
A special class of stationary flows for two-dimensional Euler equations:
A statistical mechanics description,
Commun.\ Math.\ Phys.\ 174 (1995), 229--260.
\bibitem{COS} 
D.~Chae, H.~Ohtsuka, T. Suzuki, 
Some existence results for solutions to $SU(3)$ Toda system, 
Calc.\ Var.\  (2005), 403--429.
\bibitem{CK}
S.~Chanillo, M.K.-H.~Kiessling,
Conformally invariant systems of nonlinear PDE of Liouville type,
Geometric and Functional Analysis 5, n.~6 (1995), 924--947.
\bibitem{Chavanis}
P.H.~Chavanis,
Statistical mechanics of two-dimensional vortices and stellar systems,
\textit{Dynamics and thermodynamics of Systems with Long Range Interactions,}
edited by T.~Dauxois, S.~Ruffo, E.~Arimondo and M.~Wilkens,
Lect.\ Not.\ in Phys.\ 602, Springer, Berlin, 2002.
\bibitem{cl} 
W.~Chen, C.~Li, 
Prescribing Gaussian Curvatures on Surfaces with Conical Singularities, 
J.\ of Geometric Analysis 1 n.~4 (1991), 359--372.
\bibitem{CSW}
M.~Chipot, I.~Shafrir, G.~Wolansky,
On the solutions of Liouville systems, 
J.\ Differential Equations 140 (1997), 59--105.
\bibitem{DJLW} 
W.~Ding, J.~Jost,  J.~Li, G.~Wang, 
Existence results for mean field equations,
Ann.\ Inst.\ H.~Poincar\'e, 16 n.~5 (1999),  653--666.
\bibitem{ES}
G.L.~Eyink, K.R.~Sreenivasan,
Onsager and the theory of hydrodynamic turbulence,
Reviews of Modern Physics 78 (2006), 87--135.
\bibitem{GNN}
B.~Gidas, W.M.~Ni, L.~Nirenberg,
Symmetry and Related Properties via the Maximum Principle,
Commun.\ Math.\ Phys.\ 68 (1979), 209--243.
\bibitem{JJ} 
L.~Jeanjean, 
On the existence of bounded Palais-Smale sequences and application to a 
Landesman-Lazer-type problem set on $\R^N$, 
Proc.\ Roy.\ Soc.\ Edinburgh Sect.\ A 129 (1999), 787--809.
\bibitem{JT} 
L.~Jeanjean, J.F.~Toland, 
Bounded Palais-Smale mountain-pass sequences, 
C.\ R.\ Acad.\ Sci.\ Paris S\'er. I Math.\ 327 (1998), 23--28.
\bibitem{Kiessling}
M.K.H.~Kiessling,
Statistical mechanics of classical particles with logarithmic interactions,
Comm.\ Pure Appl.\ Math.\ 46 (1993), 27--56.
\bibitem{Lin}
C.S.~Lin,
An expository survey on recent development of mean field equations,
Discr.\ Cont.\ Dynamical Systems 19 n.~2 (2007), 217--247.
\bibitem{MW} 
L.~Ma, J.~Wei,
Convergence for a Liouville equation,
Comment.\ Math.\ Helv., 76 (2001), 506--514.
\bibitem{Mo} 
J.~Moser,
A sharp form of an inequality by N. Trudinger,
Indiana Math.\ J.\ 20 (1971), 1077--1092.
\bibitem{NS}
K.~Nagasaki and T.~Suzuki,
Asymptotic analysis for two-dimensional elliptic eigenvalue problems with exponentially-dominated nonlinearities,
Asymptotic Analysis 3 (1990), 173--188.
\bibitem{ne}
C.~Neri, 
Statistical mechanics of the $N$-point vortex system with random intensities on a bounded domain,
Ann.\ Inst.\ H.~Poincar\'e Anal.\ Non Lin\'eaire 21 n.~3 (2004), 381--399.
\bibitem{ORS}
H.~Ohtsuka, T.~Ricciardi and T.~Suzuki,
Blow-up analysis for an elliptic equation describing stationary vortex flows with variable intensities in 2D-turbulence,
J.\ Differential Equations 249 n.~6 (2010), 1436--1465.
\bibitem{os3}
H.~Ohtsuka, T.~Suzuki,
Mean field equation for the equilibrium turbulence and a related functional inequality,
Adv.\ Differential Equations 11 (2006), 281--304.
\bibitem{Onsager}
L.~Onsager,
Statistical hydrodynamics, 
Nuovo Cimento Suppl.\ n.~2 6 (9) (1949), 279--287.
\bibitem{RS}
T.~Ricciardi, T.~Suzuki,
Duality and best constant for a Trudinger--Moser inequality involving probability measures,
to appear on J.\ of Eur.\ Math.\ Soc.\ (JEMS)
\bibitem{RZ}
T.~Ricciardi, G.~Zecca, 
Blow-up analysis
for some mean field equations involving probability measures from statistical hydrodynamics, 
Differential and Integral Equations 25 n.~3--4 (2012), 201--222.
\bibitem{RZ2}
T.~Ricciardi, G.~Zecca 
On Neri's mean field equation with hyperbolic sine vorticity distribution, 
Scientiae Mathematicae Japonicae  76 n.~3 (2013), 401--415.
\bibitem{R}
T.~Ricciardi,  
Mountain-Pass solutions for a Mean field equation from two-dimensional turbulence, 
Differential and Integral Equations 20 n.~5 (2007), 561--575.
\bibitem{SawadaSuzuki}
K.~Sawada, T.~Suzuki,
Derivation of the equilibrium mean field equations of point vortex and vortex filament system,
Theoret.\ Appl.\ Mech.\ Japan 56 (2008), 285--290.
\bibitem{Sq}
M.~Squassina, 
On Struwe-Jeanjean-Toland monotonicity trick, 
Proc.\ Roy.\ Soc.\ Edinburgh  A 142 (2012) 155--169.
\bibitem{St}
M.~Struwe, 
The existence of surfaces of constant mean curvature with free boundaries, 
Acta Math.\ 160 (1988), 19--64.
\bibitem{StTa}
M.~Struwe, G.~Tarantello,  
On multivortex solutions in Chern-Simons gauge theory, 
Boll.\ Unione Mat.\ Ital.\ Sez.B 1 (1998), 109--121.
\bibitem{s1book}
T.~Suzuki,
\textit{Mean Field Theories and Dual Variation,} Atlantis Press, Paris, 2009.
\bibitem{ye} 
D.~Ye, 
Une remarque sur le comportement asymptotique des solutions de $-\Delta u=\lambda f(u)$,
C.R.\ Acad.\ Sci.\ Paris 325 (1997) 1279--1282.
\end{thebibliography}
\end{document}